\documentclass[preprint,12pt]{elsarticle_delete_preprint}
\usepackage{amsfonts}

\usepackage{xcolor}
\usepackage{tikz}
\usepackage{tikz-cd}
\usepackage{bbm}
\usepackage{caption}

\usepackage{graphics}

\usepackage{amssymb}
\usepackage{amsthm}
\usepackage{amsfonts}
\usepackage{mathrsfs}
\usepackage{amscd,amssymb,amsmath,graphicx,verbatim,xy,bm,cases}
\usepackage[TS1,OT1,T1]{fontenc}
\newtheorem{theorem}{Theorem}[section]
\newtheorem{lemma}[theorem]{Lemma}
\newtheorem{corollary}[theorem]{Corollary}
\newtheorem{proposition}[theorem]{Proposition}

\theoremstyle{definition}
\newtheorem{definition}[theorem]{Definition}
\newtheorem{example}[theorem]{Example}

\theoremstyle{remark}
\newtheorem{remark}[theorem]{Remark}

\catcode`@=11
\let\@int\int \def\int{\displaystyle\@int}
\let\@lim\lim \def\lim{\displaystyle\@lim}
\let\@sum\sum \def\sum{\displaystyle\@sum}
\let\@sup\sup \def\sup{\displaystyle\@sup}
\let\@inf\inf \def\inf{\displaystyle\@inf}
\let\@cap\cap \def\cap{\displaystyle\@cap}
\let\@cup\cup \def\cup{\displaystyle\@cup}
\let\@max\max \def\max{\displaystyle\@max}
\let\@min\min \def\min{\displaystyle\@min}
\let\@frac\frac \def\frac{\displaystyle\@frac}
\let\@iint\iint \def\iint{\displaystyle\@iint}

\def\epsilon{\varepsilon}

\def\x{\bm{x}}

\numberwithin{equation}{section}

\begin{document}

\begin{frontmatter}
\title{Persistence modules induced by inner functions}

\author[a]{Jiaxing He}
\ead{hejx21@mails.jlu.edu.cn}

\author[a]{Bingzhe Hou}
\ead{houbz@jlu.edu.cn}

\author[a]{Xiao Wang}
\ead{wangxiaotop@jlu.edu.cn}

\author[b]{Yue Xin\corref{*}}
\cortext[*]{Corresponding author.}\ead{2024104@hlju.edu.cn}

\address[a]{School of Mathematics, Jilin University, 130012, Changchun, P. R. China}

\address[b]{School of Mathematical Science,
	Heilongjiang University,
	Harbin, 150080, Heilongjiang, P. R. China}

\begin{abstract}
  As well-known, inner functions play an important role in the study of bounded analytic function theory. In recent years, persistence module theory, as a main tool applied to Topological Data Analysis, has received widespread attention. In this paper, we aim to use persistence module theory to study inner functions. We introduce the persistence modules arised from  the level sets of inner functions. Some properties of these persistence modules are shown. In particular, we prove that the persistence modules (potentially not of locally finite type) induced by a class of inner functions have interval module decompositions. Furthermore, we demonstrate that the interleaving distance of  the persistence modules is continuous with respect to the supremum norm for a class of Blaschke products, which could be used to discuss the path-connectedness of Blaschke products. As an example, we provide an explicit formula for the interleaving distance of the persistence modules induced by the Blaschke products with order two.
\end{abstract}
\begin{keyword}
Inner functions; persistence modules; level sets; interleaving distance; Blaschke products.

\MSC{55N31, 30J05, 30J10, 16D70.}
\end{keyword}
\end{frontmatter}
\tableofcontents

\section{Introduction}

In this paper, we establish a connection between persistence modules and bounded analytic functions. We will first present the historical background of persistence modules and bounded analytic functions, and  then outline our main results below.

\subsection{Overview of persistence modules}

Persistent homology, as a primary tool in TDA (Topological Data Analysis), is a mathematical method for analyzing topological structures of data. In recent years,  scholars have not only conducted studies on directions such as persistence module decomposition and pseudo-metric, but also made progress in interdisciplinary theoretical intersections, such as symplectic geometry and Hamiltonian systems \cite{Viterbo-2022,UZ-2016}, fractal dimensions \cite{Sch-2020,AAFKMNPS-2018}, etc.  Additionally, persistent homology has found successful applications in fields like computational biology \cite{WCW-2020, ACCGRV-2022},  pattern recognition \cite{LFMC-2023, Car-2014}, computer graphics \cite{PSO-2018, COO-2015}, etc.

Persistence modules are the basic algebraic objects  in the study of persistent homology. One of important properties of persistence modules is whether it has an interval module decomposition. Assuming that a persistence module is a direct sum of interval modules, it is found that the collection of such intervals, the barcode, is  useful to classify data. Since Zomorodian and Carlsson  \cite{Zo-Gu-2005} introduced persistence modules of finite type indexed by natural numbers and proved that they can be  decomposed into interval modules,  there are many works concerning the interval module decomposition of persistence modules of finite type \cite{CV-2010, BLO-2022}, and of locally finite type \cite{CB-2015, Bot-2017,BG-2022} in recent years. 


The interleaving distance, an algebraic generalization of the bottleneck distance, is a useful pseudo-metric on both persistence modules and multiparameter persistence modules. This concept was first introduced by Chazal et al. in \cite{CDGO-2009}. Recently there are some works concerning the existence of $\epsilon$-interleavings which is important to compare two persistence modules,  such as \cite{TP-2024,BG-2022}. Since the algebraic stability theorem was initially proved by Chazal et al. in \cite{CDGO-2009, CDGO-2016}, there are several researchers study the stability of interleaving distance on persistence modules, such as \cite{UB-2015,BL-2018,ML-2015,Blum-Les-2023}.

\subsection{Overview of inner functions}

The topological problems in the space of analytic functions are important subjects that have attracted wide attention, for example, the Corona problem. Let $\mathbb{D}$ be the unit open disk in the complex plane $\mathbb{C}$ and let $\partial\mathbb{D}$ be the boundary of $\mathbb{D}$, i.e., the unit circle. A bounded analytic function $f$ on $\mathbb{D}$ is called an inner function if it has unimodular radial limits almost everywhere on the boundary $\partial\mathbb{D}$ of $\mathbb{D}$. Inner functions play a central role in many important results such as Beurling Theorem.

In the research of topological behaviors in the space of inner functions,  it is useful to study the level sets of inner functions.  Berman \cite{Berman-1984},  Stephenson and  Sundberg \cite{KC1985} conducted  researches on the level sets of inner functions, they proved that two inner functions which share a common level set for some value $r$ in $(0, 1)$ must agree up to a unimodular constant.   Bickel and  Gorkin \cite{KB} posed a version of the conjecture in this setting, called the level set Crouzeix (LSC) conjecture, and established structural and uniqueness properties for (open) level sets of finite Blaschke products that the LSC conjecture can be proved in several cases. Besides, through the concept of level sets,  Cohn \cite{BC} proposed a class of inner functions named one-component inner functions.  Furthermore,  Cohn characterized the Carleson measures for the model spaces $H^2\ominus uH^2$, under the assumption that $u$ is a one-component inner function.  Cima and  Mortini \cite{CM-2017, CM-2020} conducted a further study on the properties of one-component inner functions and provided some equivalent characterizations of one-component inner functions.

Interpolating Blaschke products are a class of important Blaschke products, which are closely related to the Interpolating Problem and Corona Problem. Following from a celebrated result of
Carleson \cite{Car-Len-1958}, one can see that a Blaschke product is interpolating if and only if
its zero points are  uniformly separated. So far, interpolating Blaschke products have received widespread attention. However, there is still an open problem that whether the set of all interpolation Blaschke products is dense in the inner function space.

The path-connectedness in the space of inner functions is an important research subject.  Nestoridis studied the invariant and noninvariant connected components of the inner functions space $\mathcal{F}$. He proved that the inner functions $d(z)={\rm exp} \{(z+1)/(z-1)\}$ and $zd$ belong to the same connected component \cite{Ne79}. By imposing restrictions on the diameters of the connected components of the level sets,  Nestoridis gave a family of inner functions, denoted by $H$, such that for every $B\in H$, $B$ and $zB$ don't belong to the same component \cite{Ne80}.

Through the level sets of functions, we  are able to establish a connection between bounded analytic function theory and persistent homology theory. Based on the definition of the persistence modules induced by inner functions, we can analyze the topological properties of inner functions  by observing the persistence modules and their corresponding barcodes.



\subsection{Statements of main results}
In this subsection, we outline our main results and leave the proofs in the later sections.

First, we introduce the definition of persistence modules induced by inner functions and their properties. Consider an inner function $u(z)$, we define the $\theta$-level set of $u$ as follows,
\begin{equation*}
	\Omega_{u,\theta}:=\{z\in\mathbb{D};|u(z)|<\theta\}.
\end{equation*}
Naturally, the parameter $\theta$ induces a filtration on $\mathbb{D}$, i.e., $\{\Omega_{u,\theta}\}_{\theta\in (0, 1)}$ and natural inclusions $i_{\eta, \theta}: \Omega_{u,\eta} \hookrightarrow \Omega_{u,\theta}$ for any $0<\eta\leq\theta<1$. 
For technical reasons, we would like to introduce a change of variable, for any $\theta\in (0, 1)$, let
\[
t(\theta)=\ln\frac{1+\theta}{1-\theta}.
\]
We also write
\begin{equation*}
	\Sigma_{u,t(\theta)}:=\Omega_{u,\theta}=\{z\in\mathbb{D};|u(z)|<\theta\},
\end{equation*}
and when parameter is obvious from context, we write $t$ for $t(\theta)$. Also we write $\theta = \theta(t)=\frac{\mathrm{e}^t-1}{\mathrm{e}^t+1}$ to be the inverse function of $t(\theta)$.

Naturally, the parameter $t$ also induces a filtration on $\mathbb{D}$, i.e., $\{\Sigma_{u,t}\}_{t\in (0, +\infty)}$ and natural inclusion  $i_{s,t}: \Sigma_{u,s} \hookrightarrow \Sigma_{u,t}$ for any $0<s\leq t<+\infty$. 

For an arbitrary fixed field $\mathbb{F}$, let $V_t^{|u|}=\oplus_{k}(V_t^{|u|})_k$   be the  homology groups  of the above sublevel sets $\Sigma_{u,t}$ for $t\in (0, +\infty)$, that is,
\[
(V_t^{|u|})_k:=H_{k}(\Sigma_{u,t}; \mathbb{F}).
\] 
For any $0<s\leq t<+\infty$, the natural inclusion $i_{s,t}: \Sigma_{u,s} \hookrightarrow \Sigma_{u,t}$ induces the homomorphism 
\[
\pi_{s,t}:=(i_{s,t})_{k}:(V_s^{|u|})_k \to (V_t^{|u|})_k.
\] 
Then we get a persistence module $(V^{|u|},\pi)$, where $V^{|u|}$ is  the collection of $\{V_t^{|u|}\}_{t\in (0,+\infty)}$.  For convenience, we write   the persistence module induced by the level sets of inner function  $u$ as $\mathbb{U}=(V^{|u|},\pi)$.

Now, let us see some fundamental properties of the components of level sets for inner functions.
The following statement $(1)$ is "well known", to our best knowledge, it first appears in Tsuji's book \cite[Theorem VIII]{Ts-1975} and one can find a proof in \cite{CM-2017}.

\begin{lemma}[\cite{Ts-1975}, \cite{CM-2017}]\label{sc-components}
	Given a non-constant inner function $u$ in $H^{\infty}$ and $\eta\in(0,1)$, let
	$\Omega := \Omega_{u}(\eta) = \{z\in\mathbb{D}; |u(z)| < \eta\}$ be a level set. Suppose that $\Omega_0$ is a component (= maximal connected subset) of $\Omega$. Then
	\begin{enumerate}
		\item $\Omega_0$ is a simply connected domain; that is, $\mathbb{C}\setminus\Omega_0$ has no bounded components.
		\item $\inf_{\Omega_0} |u| = 0$.
	\end{enumerate}
\end{lemma}
By Lemma \ref{sc-components}, we have an immediate corollary,
\begin{corollary}\label{HB}
	Let $u(z)$ be an inner function. For any $t\in (0,+\infty)$, we have $H_{i}(\Sigma_{u,t})=0$, $i\geq 1$.
\end{corollary}

By Corollary \ref{HB}, we only need to focus on the zeroth homology groups  of the sublevel sets, i.e., $V^{|u|}=(V^{|u|})_0$.

In Section $3$,  we will show some fundamental properties of the persistence modules induced by inner functions. For example, the persistence modules induced by inner functions are invariant under M\"{o}bius transformations. More precisely, for any inner function $u$, if $v=u\circ\varphi$ is the composition of $u$ and a M\"{o}bius transformation $\varphi$, then  the persistence modules induced by $u$ and $v$ are isomorphic.

In analytic function theory, some special classes of inner functions play an important role. In  \cite{Berman-1984}, Berman characterized a class of inner functions through radical limits. When studying the connectedness of inner functions, Nestoridis \cite{Ne80} introduced a class of Blaschke products in noninvariant connected components. We extend the conditions  given by Berman and Nestoridis as follows.

\begin{definition}
	Let $u$ be an inner function,
	\begin{enumerate}
		\item (\textbf{Property} \bm{$\mathfrak{B}$})  We say that $u$ has Property \bm{$\mathfrak{B}$}, if for any $\theta\in (0,1)$,  the closure of each component of the level set $\Omega_{u,\theta}$ is contained in the unit open disk $\mathbb{D}$.
		\item (\textbf{Weak Property} \bm{$\mathfrak{B}$})  We say that $u$ has weak Property \bm{$\mathfrak{B}$}, if there exists $\theta\in(0,1)$ such that the closure of each component of the level set $\Omega_{u,\theta}$ is contained  in the unit open disk $\mathbb{D}$. Moreover, we call such an inner function $u$ is of \bm{$\theta$}-weak Property \bm{$\mathfrak{B}$}.
		\item (\textbf{Strong Property} \bm{$\mathfrak{B}$})  We say that $u$ has strong Property \bm{$\mathfrak{B}$}, if $u$ has Property \bm{$\mathfrak{B}$} and there exists $\eta\in(0,1)$ such that $\delta_{u,\eta}<1$, where
			\[
			\delta_{u,\eta}=\sup\{\rho(z_1, z_2);~z_1 \ \text{and} \ z_2 \ \text{belong to the same component of} \ \Omega_{u,\eta}\},  
			\]
		and $\rho(z_1, z_2) = |\frac{z_1-z_2}{1-\overline{z_1}z_2}|$ is the pseudo-hyperbolic distance between $z_1$ and $z_2$.
		Moreover, we call such an inner function $u$ is of \bm{$\eta$}-strong Property \bm{$\mathfrak{B}$}.
		\item (\textbf{Property} \bm{$\mathfrak{H}$}) We say that $u$ has  Property \bm{$\mathfrak{H}$}, if $\delta_{u,\eta}<1$  for every $\eta\in(0,1)$.
		\item (\textbf{Weak Property} \bm{$\mathfrak{H}$})  We say that $u$ has weak Property \bm{$\mathfrak{H}$}, if $u$ has weak Property \bm{$\mathfrak{B}$} and there exists $\eta\in(0,1)$, such that $\delta_{u,\eta}<1$.  Moreover, we call such an inner function $u$ is of \bm{$\eta$}-weak Property \bm{$\mathfrak{H}$}.
	\end{enumerate}
\end{definition}
\begin{figure}[h]
	\centering
	\includegraphics[width=1\textwidth]{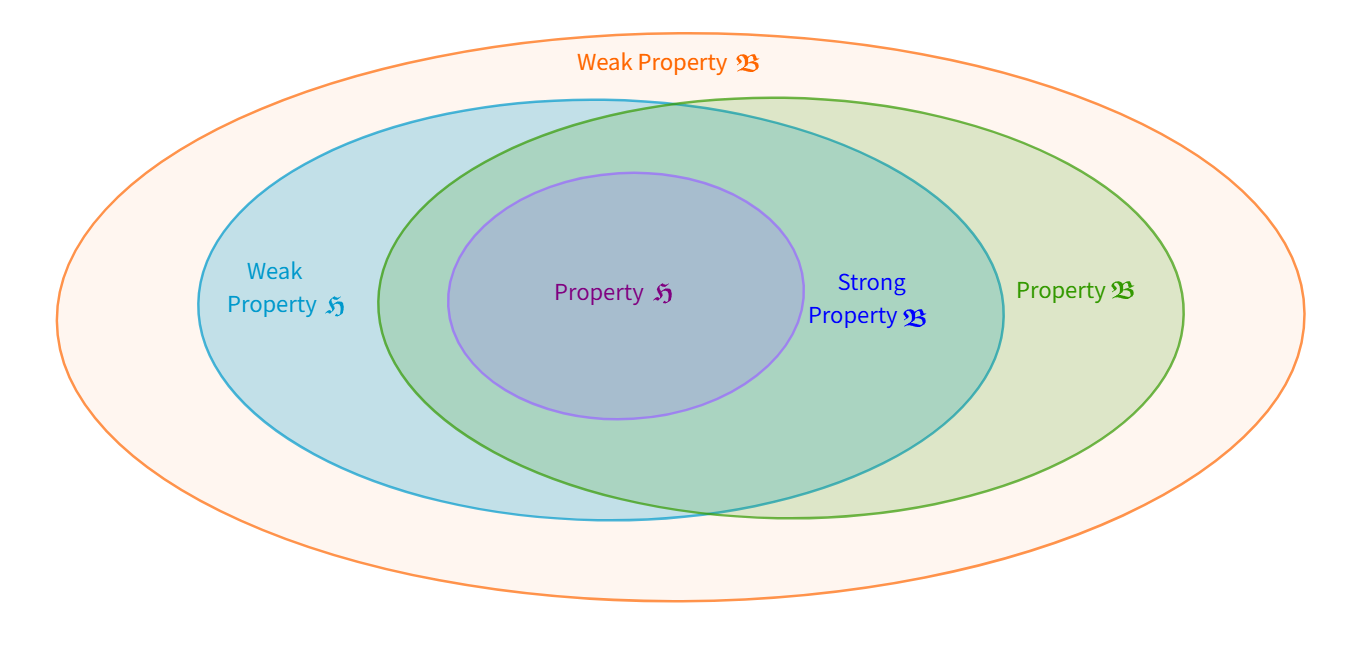}
	\caption{Relationship}
	\label{fig:set}
\end{figure}

The relationship between those properties is demonstrated in Figure \ref{fig:set}.
We can show that each inner function with weak Property \bm{$\mathfrak{B}$} is a Blaschke product, each interpolating Blaschke product has  weak Property \bm{$\mathfrak{H}$} and each Blaschke product with  weak Property \bm{$\mathfrak{H}$} is a Carleson-Newman Blaschke product. In this article, we mainly focus on the persistence modules induced by Blaschke products with those properties.

 In Section $3$, we will give a topological description of the critical points of inner functions $u$ with Property \bm{$\mathfrak{B}$}, more precisely, for a critical point $w$ with order $m$, there will be $m+1$ components merging at $w$. Then we obtain the interval module decompositions of the associate persistence modules as in the following  Theorem A.
Let  $\mathcal{Z}(u)$ be the set of zero points of $u$. 

\vspace{2mm}

\noindent \textbf{Theorem A.} \ 
Suppose that $B$ is a Blaschke product with property $\mathfrak{B}$. For each $\omega_j\in \mathcal{Z}(B')\setminus\mathcal{Z}(B)$, denote the order of $\omega_{j}$ by $m_j$, i.e., $B^{(k)}(\omega_{j})=0$, $k=1,\ldots,m_{j}$  and $B^{(m_j+1)}(w_j)\neq 0$. The persistence module $\mathbb{B}$ induced by $B$ can be decomposed into interval modules,
$$
\mathbb{B}\cong\mathbb{F}(0,+\infty)\oplus(\oplus_{j} (\mathbb{F}(0,s_j])^{m_{j}}),
$$
where $s_j=\ln\frac{1+|B(w_j)|}{1-|B(w_j)|}$.
\vspace{2mm}

Notice that there may be infinitely many $w_{j}$ corresponding to the same critical value, which implies that the associated persistence modules are potentially not of locally finite type.

In Section $4$, we study the relationship between  the distance induced by supremum norm $\|\cdot\|_{\infty}$ on inner functions with strong Property $\mathfrak{B}$ and the interleaving distance $d_{int}$ on the induced persistence modules.

\vspace{2mm}

\noindent \textbf{Theorem B.} \ 
Suppose that $B(z)$ is a Blaschke product with strong Property $\mathfrak{B}$. Then, for any $\epsilon>0$, there exists an $\eta\in(0,1)$ such that for any Blaschke product $\widetilde{B}(z)$ with strong Property $\mathfrak{B}$,  $\|\widetilde{B}(z)-B(z)\|_{\infty}<\eta$ implies $d_{int}(\mathbb{B},\widetilde{\mathbb{B}})<\epsilon $. 
\vspace{2mm}

Theorem B means  the interleaving distance $d_{int}$ is continuous with respect to  the  distance induced by supremum norm for Blaschke products  with strong Property $\mathfrak{B}$.
To prove Theorem B, we provide a series of lemmas. Those lemmas are not only useful when proving Theorem B, but also valuable in analytic function theory. For instance, for any  Blaschke product  $B(z)$ with strong Property $\mathfrak{B}$, $\delta_{B, \theta}\rightarrow 0$ as $\theta\rightarrow 0$. Besides, the interleaving distance also could be used to study the problems in analytic function theory. For instance, if the interleaving distance between the persistence modules induced by two Blaschke products with strong Property $\mathfrak{B}$ is infinite, then the two Blaschke products can not be connected by any path in the family of Blaschke products with strong Property $\mathfrak{B}$.

In the last section, we provide an explicit formula for the interleaving distance of the persistence modules induced by  Blaschke products with order two.  Let $B_{1}$ be the Blashcke product with two zero points $\beta_{0}$ and $\beta_{1}$, and let $B_{2}$ be the Blashcke product with two zero points with $\gamma_{0}$ and $\gamma_{1}$. Then, 
\[
d_{int}(\mathbb{B}_{1},\mathbb{B}_{2})=\min\left(\max(\frac{1}{2}\ln\frac{1}{\sqrt{1-|w_{2}|^{2}}},\frac{1}{2}\ln\frac{1}{\sqrt{1-|w_{1}|^{2}}}),{\big \vert}\ln\frac{\sqrt{1-|w_{1}|^{2}}}{\sqrt{1-|w_{2}|^{2}}}{\big \vert}\right),
\] 
where $w_{1}=\frac{\beta_{0}-\beta_{1}}{1-\overline{\beta_{0}}\beta_{1}}$, $w_{2}=\frac{\gamma_{0}-\gamma_{1}}{1-\overline{\gamma_{0}}\gamma_{1}}$.

\section{Preliminaries}

In this section, we will introduce some definitions related to  persistence modules and Blaschke products. Firstly, we provide the following definitions of persistence modules (compare \cite{ carlsson2009theory, PRSZ-2020, CDGO-2016}). Let $\mathbb{F}$ be an arbitrary field.

\begin{definition}\label{def_pm}
	A persistence module is a pair $(V,\pi)$, where $V$ is a collection $\{V_{t};t\in\mathbb{R}\}$ of vector spaces over $\mathbb{F}$, and $\pi$ is a collection $\{ \pi_{s,t}\}$ of linear maps $\pi_{s,t}: V_{s}\rightarrow V_{t}$ for all $s\leq t$ in $\mathbb{R}$ such that (1), (2) and (3) below hold,
	\enumerate{
	\item[(1)] ($Persistence$) For any $s\leq t\leq r$ one has $\pi_{s,r}=\pi_{t,r}\circ\pi_{s,t}$, i.e., the following diagram commutes:
	\begin{center}
		\begin{tikzpicture}
			\node[inner sep=1pt] (a) at (0,0) {$V_{s}$};
			\node[inner sep=1pt] (b) at (2,0) {$V_{t}$};
			\node[inner sep=1pt] (c) at (4,0) {$V_{r}$};
			\node at (1,-0.25) {$\pi_{s,t}$};
			\node at (2,0.85) {$\pi_{s,r}$};
			\node at (3,-0.25) {$\pi_{t,r}$};
			\draw[-latex] (a.0) -- (b.180);
			\draw[-latex] (b.0) -- (c.180);
			\draw[-latex] (a.40) to[out = 30,in=150]  (c.150);
		\end{tikzpicture}
	\end{center}
	\item[(2)] $\pi_{t,t}=\mathbbm{1}_{V_{t}}$ for any $t\in \mathbb{R}$, where $\mathbbm{1}_{V_{t}}$ represents the identity map on $V_t$.
	\item[(3)] There exits an $s_{-}\in \mathbb{R}$, such that $V_{s}=0$ for any $s\leq s_{-}$.
	
	We denote $(V,\pi)$ by $\mathbb{V}$ in brief.
	}
\end{definition}

	A point $t\in\mathbb{R}$ is called spectral for a persistence module $(V,\pi)$ if for any neighborhood $U\ni t$ there exists $s<r$ in $U$, such that $\pi_{s,r}:V_s\rightarrow V_r$ is not an isomorphism. Denote by $\text{Spec}V=\text{Spec}(V,\pi)$ the collection of spectral points of $(V,\pi)$ together with $+\infty$. This set will be called the spectrum of $\mathbb{V}$.

\begin{remark}
	A persistence module $(V,\pi)$ is called of locally finite type if  $V$ is a collection $\{V_t;t\in\mathbb{R}\}$ of finite-dimensional vector spaces over $\mathbb{F}$ and the following conditions (a) and (b) also hold.
	\enumerate{
	\item[(a)] The spectrum of $\mathbb{V}$ is a closed discrete bounded from below subset of $\mathbb{R}$ (but not necessarily finite).
	\item[(b)] ($Semicontinuity$) For any $t\in\mathbb{R}$ and any $s\leq t$ sufficiently close to $t$, the map $\pi_{s,t}$ is an isomorphism.
	}
\end{remark}

\begin{remark}
	A persistence module $(V,\pi)$ is called of finite type if  $V$ is a collection $\{V_t;t\in\mathbb{R}\}$ of finite-dimensional vector spaces over $\mathbb{F}$ and the following conditions (c)  also hold.
	\enumerate{
		\item[(c)] For all but a finite number of points $t\in\mathbb{R}$ there exists a neighborhood $U$ of $t$ such that $\pi_{s,t}$ is an isomorphism for any $s<t$ in $U$.
	}
\end{remark}

Let $(V,\pi)$, $(V',\pi')$ be two persistence modules, 
\begin{definition}
	A (persistence) morphism $\varphi:(V,\pi)\rightarrow(V',\pi')$ is a family of linear maps $\varphi_t:V_{ t} \rightarrow V_{ t}'$ such that the following diagram commmutes for all $s\leq t$,
	\begin{center}
		\begin{tikzcd}
			V_{s} \arrow{d}{\varphi_{s}} \arrow{r}{\pi_{s,t}} & V_{ t} \arrow{d}{\varphi_{t}}\\
			V_{ s}' \arrow{r}{\pi_{s,t}'}         & V_{ t}'
		\end{tikzcd}.
	\end{center}
\end{definition}

Two persistence modules $(V,\pi)$ and  $(V',\pi')$ are isomorphic if there exist two morphisms $\varphi:\mathbb{V}\rightarrow \mathbb{V}'$ and $\psi :\mathbb{V}'\rightarrow \mathbb{V}$ such that both compositions $\psi\circ\varphi$ and $\psi\circ\varphi$ are the identity morphisms on the corresponding persistence modules, where the identity morphism on $\mathbb{V}$ is the identity on $V_t$ for all t.

Let $(V,\pi)$ be a persistence module and $\Delta \in \mathbb{R}$, define a persistence module $(V[\Delta],\pi[\Delta])$ by taking $(V[\Delta])_t=V_{t+\Delta}$ and $(\pi[\Delta])_{s,t}=\pi_{s+\Delta,t+\Delta}$. This new persistence module is called the $\Delta$-shift of $V$. For $\Delta>0$, the map $\Phi^{\Delta}:(V,\pi)\rightarrow(V[\Delta],\pi[\Delta])$ defined by $\Phi^{\Delta}_{t}=\pi_{t,t+\Delta}$ is a morphism of persistence modules (it will be called a $\Delta$-shift morphism). Also, if we have a morphism $F:\mathbb{V}\rightarrow \mathbb{W}$ between two persistence modules, we denote by $F[\Delta]:\mathbb{V}[\Delta]\rightarrow \mathbb{W}[\Delta]$ the corresponding morphism between their $\Delta$-shifts.

\begin{definition}
	Let $(V,\pi)$ be a persistence module. A persistence submodule  $(W,\tilde{\pi})$ of $\mathbb{V}$ is a collection of subspaces $W_{s}\subseteq V_{s}$ for all $s\in \mathbb{R}$, such that the maps $\tilde{\pi}_{s,t}:=\pi_{s,t}|_{W_{s}}:W_{s}\rightarrow W_{t}$ are well-defined for all $s<t$, and yield a persistence module $(W,\tilde{\pi})$. 
\end{definition}

\begin{definition}
	Let $(V,\pi)$, $(V',\pi')$ be two persistence modules. Their direct sum $(W,\theta)$ is the persistence module whose underlying modules are $W_{t} = V_{t}\oplus V'_{t}$ and accordingly, $\theta_{s,t}=\pi_{s,t}\oplus\pi'_{s,t}$.
\end{definition}

As an basic example of persistence modules, we  introduce the concept of interval module.

\begin{example}(Interval modules)
	For an interval $(a,b]$ (with $b< +\infty$), define a persistence module $\mathbb{F}(a,b]$ as follows:
	$$
	\mathbb{F}(a,b]_{t}=\left\{
	\begin{array}{rcl}
		\mathbb{F}       &   &   {\text{if}\ t\in (a,b]},\\
		0       &   &   {\text{otherwise}},
	\end{array}
	\right.	
	\pi_{s,t}=\left\{
	\begin{array}{rcl}
		\mathbbm{1}       &   &   {\text{if}\ s,t\in (a,b]},\\
		0       &   &   {\text{otherwise}}.
	\end{array}
	\right.
	$$
	For an interval $(a,+\infty)$, define a persistence module $\mathbb{F}(a,+\infty)$ as follows:
	$$
	\mathbb{F}(a,+\infty)_{t}=\left\{
	\begin{array}{rcl}
		\mathbb{F}       &   &   {\text{if}\ t\in (a,+\infty)},\\
		0       &   &   {\text{otherwise}},
	\end{array}
	\right.	
	\pi_{s,t}=\left\{
	\begin{array}{rcl}
		\mathbbm{1}       &   &   {\text{if}\ s,t\in (a,+\infty)},\\
		0       &   &   {\text{otherwise}}.
	\end{array}
	\right.
	$$ Such persistence modules will be called interval modules. 
\end{example}

Moreover, a persistence module is indecomposable if it is not isomorphic to a direct sum of two non-trivial persistence modules. In particular, interval modules are indecomposable. For more details, we refer to \cite{PRSZ-2020,CDGO-2016}.


\begin{definition}
	Given $\Delta>0$, we say that two persistence modules $\mathbb{V}$ and $\mathbb{W}$ are $\Delta$-interleaved if there exist two morphisms $F:\mathbb{V}\rightarrow \mathbb{W}[\Delta]$ and $G: \mathbb{W}\rightarrow \mathbb{V}[\Delta]$, such that the following diagrams commute:
	\begin{center}
		\begin{tikzpicture}
			\node[inner sep=1pt] (a) at (0,0) {$\mathbb{V}$};
			\node[inner sep=1pt] (b) at (2,0) {$\mathbb{W}[\Delta]$};
			\node[inner sep=1pt] (c) at (4,0) {$\mathbb{V}[2\Delta]$};
			\node at (1,-0.25) {$F$};
			\node at (2,0.95) {$\Phi_{\mathbb{V}}^{2\Delta}$};
			\node at (3,-0.25) {$G[\Delta]$};
			\draw[-latex] (a.0) -- (b.180);
			\draw[-latex] (b.0) -- (c.180);
			\draw[-latex] (a.40) to[out = 30,in=150]  (c.160);
		\end{tikzpicture}\ \ \
		\begin{tikzpicture}
			\node[inner sep=1pt] (a) at (0,0) {$\mathbb{W}$};
			\node[inner sep=1pt] (b) at (2,0) {$\mathbb{V}[\Delta]$};
			\node[inner sep=1pt] (c) at (4,0) {$\mathbb{W}[2\Delta]$};
			\node at (1,-0.25) {$G$};
			\node at (2,0.95) {$\Phi_{\mathbb{W}}^{2\Delta}$};
			\node at (3,-0.25) {$F[\Delta]$};
			\draw[-latex] (a.0) -- (b.180);
			\draw[-latex] (b.0) -- (c.180);
			\draw[-latex] (a.40) to[out = 30,in=150]  (c.160);
		\end{tikzpicture}
	\end{center}
	where $\Phi_{\mathbb{V}}^{2\Delta}$ and $\Phi_{\mathbb{W}}^{2\Delta}$ are the shift morphisms. We also refer to such a pair of morphisms $F$ and $G$ as $\Delta$-interleaving morphisms.
\end{definition}

\begin{definition}
	For two persistence modules $\mathbb{V}$ and $\mathbb{W}$, we define the interleaving distance between them to be
	$$
	d_{int}(\mathbb{V},\mathbb{W})=\text{inf}\{\Delta>0\ |\ \mathbb{V}\  \text{and} \ \mathbb{W} \  \text{are} \  \Delta-\text{interleaved}\}.
	$$
\end{definition}

Note that $d_{int}$  is a pseudo-metric on the set of persistence modules \cite{PRSZ-2020}. The following example shows that
$d_{int}$ is not a metric.
\begin{example}[\cite{ML-2015}]
	Let $\mathbb{M}$ be the 1-module\footnote{Here the 1-module means the single parameter persistence modules. Moreover, Lesnick in \cite{ML-2015} considered n-dimensional persistence modules.} with $\mathbb{M}_{0}=F$ and $\mathbb{M}_{a}=0$ if $a\neq 0$. Let $\mathbb{N}$ be the trivial 1-module. 
\end{example}

\begin{remark}
	Then $\mathbb{M}$ and $\mathbb{N}$ are not isomorphic, and so are not $0$-interleaved, but it is easy to check that $\mathbb{M}$ and $\mathbb{N}$ are $\Delta$-interleaved for any $\Delta>0$. Thus, $d_{int}(\mathbb{M},\mathbb{N})=0$.
	
\end{remark}

We next introduce the bottleneck distance $d_{bot}$ on the space of barcodes. A barcode $\mathcal{B}=\{(I_i,m_i)\}$ is a family of intervals $I_{i}$ with given multiplicities where $m_i$ can be infinite, where each interval $I_i$ is either  $(a,b]$ or  $(a,+\infty)$. The intervals in a barcode are sometimes called bars.

Given an interval $I=(a,b]$, denote by $I^{-\Delta}=(a-\Delta,b+\Delta]$ the interval obtained from $I$ by expanding by $\Delta$ on both sides. Let $\mathcal{B}$ be a barcode. For $\epsilon>0$, denoted by $\mathcal{B}_{\epsilon}$ the set of all bars from $\mathcal{B}$ of length greater than $\epsilon$.

A matching between two finite multi-sets $X$, $Y$ is a bijection $\mu:X'\rightarrow Y'$, where $X'\subset X$, $Y'\subset Y$. In this case, $X'=\text{coim}\mu$, $Y'=\text{im}\mu$, and we say that elements of $X'$ and $Y'$ are matched. If an element appears in the multi-set several times, we treat its different copies separately, e.g., it could happen that only some of its copies are matched.

\begin{definition}
	A $\Delta$-matching between two barcodes $\mathcal{B}$ and $\mathcal{C}$ is a matching $\mu:\mathcal{B}\rightarrow\mathcal{C}$ such that:\\
	\quad(1) $\mathcal{B}_{2\Delta}\subset\text{coim}\mu$,\\
	\quad(2) $\mathcal{C}_{2\Delta}\subset\text{im}\mu$,\\
	\quad(3) If $\mu(I)=J$, then $I\subset J^{-\Delta}$, $J\subset I^{-\Delta}$.
	
	$d_{bot}(\mathcal{B},\mathcal{C})$ between two barcodes $\mathcal{B}$ and $\mathcal{C}$ is defined to be the infimum over all $\Delta$ for which there is a $\Delta$-matching between $\mathcal{B}$ and $\mathcal{C}$.
\end{definition}

 We denote by $\mathcal{B}(\mathbb{V})$ the barcode corresponding to a persistence module $\mathbb{V}$. To compute the distance between the persistence modules induced by finite Blaschke products, we introduce the following isometry theorem about the persistence module of finite type.

\begin{theorem}[\cite{PRSZ-2020}]
	The map $\mathbb{V}\mapsto\mathcal{B}(\mathbb{V})$ is an isometry, i.e., for any two  persistence modules of finite type $\mathbb{V}$, $\mathbb{W}$, we have $d_{int}(\mathbb{V},\mathbb{W})=d_{bot}(\mathcal{B}(\mathbb{V}),\mathcal{B}(\mathbb{W}))$.
\end{theorem}

\begin{remark}\cite{PRSZ-2020}\label{twodi} Fix $a,b,c,d<\infty$, with $a<b$, $c<d$, and consider $d_{int}(\mathbb{F}(a,b],\mathbb{F}(c,d])$ between the persistence modules $\mathbb{F}(a,b]$ and $\mathbb{F}(c,d])$, then we have
	\begin{align*}
	d_{int}(\mathbb{F}(a,b],\mathbb{F}(c,d]) =&d_{bot}((a,b],(c,d])\\
	=&\min\left(\max(\frac{b-a}{2},\frac{d-c}{2}),\max(|a-c|,|b-d|)\right).
	\end{align*}
\end{remark}

The following theorem is rooted in Azumaya-Krull-Remak-Schmidt Theorem \cite{AKRS-1950}.
\begin{theorem}[\cite{CDGO-2016}]\label{unique}
	Suppose that a persistence module $\mathbb{V}$ over
	$\mathbf{T}\subseteq \mathbf{R}$ can be expressed as a direct sum of interval modules in two different ways,
	$$
	\mathbb{V}=\bigoplus_{l\in L} {\mathbb{I}}^{J_{l}}=\bigoplus_{m\in M} {\mathbb{I}}^{K_{m}}.
	$$
	Then there is a bijection $\sigma\ :\ L\rightarrow M$ such that $J_{l}=K_{\sigma(l)}$ for all $l$.
\end{theorem}

In the above theorem, $J\subseteq \mathbf{T}$ is an interval and $\mathbb{I}^{J}$  is the same as our interval module $\mathbb{F}(J)$. 

\begin{remark}
	Crawley-Boevey \cite{CB-2015} has recently shown that a persistence module over $\mathbf{R}$ admits
	an interval module decomposition if each $V_{t}$
	is finite-dimensional. Thus, if $\mathbf{T}$ is a locally finite subset  in $\mathbf{R}$, and each $V_{t}$ is finite-dimensional, a persistence module over $\mathbf{T}$ also admits
	an interval module decomposition for all $\mathbf{T}\subseteq\mathbf{R}$.
\end{remark}

Next, we will introduce several important theorems of inner functions which are useful in this paper.

\begin{definition}
	A Blaschke product is an inner function of the form
	\begin{equation*}
		B(z)=\lambda z^{m}\prod_{n}\frac{|\alpha_{n}|}{\alpha_{n}}(\frac{\alpha_{n}-z}{1-\overline{\alpha_{n}}z}),
	\end{equation*}
	where $m$ is a nonnegative integer, $\lambda$ is a complex number with $|\lambda|=1$, and $\{\alpha_{n}\}$ is a sequence of points in $\mathbb{D}\setminus\{0\}$ satisfying the Blaschke condition $\sum_{n}(1-|\alpha_{n}|)<\infty$.
\end{definition}

If for every bounded sequence of complex numbers $\{w_n\}_{n=1}^{\infty}$, there exists $f$ in $H^{\infty}$ satisfying $f(\alpha_n)=w_n$ for every $n\in\mathbb{N}$, then both the sequence $\{\alpha_n\}_{n=1}^{\infty}$ and the Blaschke product $B(z)$ are called interpolating. Following from a celebrated result of Carleson \cite{Car-Len-1958}, one can see
that $B(z)$ is an interpolating Blaschke product if and only if $\{\alpha_{n}\}$ is a uniformly separated sequence, i.e.,
\begin{equation*}
	\inf_{n\in \mathbb{N}}\prod_{k\neq n}{\big \vert}\frac{\alpha_{k}-\alpha_{n}}{1-\overline{\alpha_{k}}\alpha_{n}}{\big \vert}>0.
\end{equation*}
Furthermore, a Blaschke product is called Carleson-Newman if it is a product of finitely many interpolating Blaschke products.

The result proposed by Hoffman \cite{Hoff-1967} (see also \cite{Gar-1981}) plays an important role in the research of interpolating Blaschke products. For convenience, we show the representation of the Hoffman's Lemma by Cima and Motini \cite{CM-2017}.
\begin{lemma}[Hoffman's Lemma, \cite{Hoff-1967}]\label{Hoffman}
	Let $\delta$, $\eta$ and $\epsilon$ be real numbers, called Hoffman constant, satisfying
	 $0 <\delta <1$, $0<\eta < (1-\sqrt{1-\delta^2})/\delta$, (that is, $0<\eta<\rho(\delta, \eta)$) and
	\[
	0<\epsilon<\eta\cdot\frac{\delta-\eta}{1-\delta\eta}.
	\]
	If $B$ is any interpolating Blaschke product with zeros $\{z_n; n \in \mathbb{N}\}$ such that
	\[
	\delta(B)=\inf\limits_{n\in\mathbb{N}} (1-\vert z_n\vert^2)\vert B'(z_n) \vert \geq \delta,
	\]
	then
	\begin{enumerate}
	\item The pseudo-hyperbolic disks $D_{\rho}(a, \eta)$ for $a \in \mathcal{Z}(B)$ are pairwise disjoint.
	\item The following inclusions hold:
	\[
	\{z\in\mathbb{D}; \vert B(z) \vert<\epsilon\} \subseteq \{z\in\mathbb{D}; \rho(z, \mathcal{Z}(B))<\eta\} \subseteq \{z\in\mathbb{D}; \vert B(z) \vert<\eta\}.
	\]
	\end{enumerate}

\end{lemma}

Berman \cite{Berman-1984} characterized a class of inner functions whose components of level sets equip with closures in $\mathbb{D}$.

\begin{theorem}[Berman]\label{finitecomponet}
	Let $u$ be an inner function. Then, for every $\epsilon\in(0,1)$, all the
	components of the level sets $\{z\in\mathbb{D}; |u(z)| < \epsilon\}$ have compact closures in $\mathbb{D}$ if and only if
	$u$ is a Blaschke product and
	\[
	\limsup\limits_{r\rightarrow1}|u(r\xi)|=1 \ \ \  \text{for every} \ \  \xi\in\mathbb{T}.
	\]
\end{theorem}

 Nestoridis \cite{Ne80} introduced a family of inner functions $B$, denoted by $H$ (in our language, $H$ is the family of inner functions with Property $\mathfrak{H}$).  For every $B\in H$, $B$ and $zB$ don't belong to the same path-connected component.  
The elements in $H$ are all Blaschke products including all of thin Blashcke products.

\section{Persistence modules induced by inner functions with Property $\mathfrak{B}$ and interval module decompositions}

In this section, we consider the inner functions with property $\mathfrak{B}$. First, we consider the relation between the order of the critical points and the number of connected components merging at those points. And then we decompose the persistence modules induced by Blaschke products with Property $\mathfrak{B}$ into interval modules.

\begin{lemma}\label{keeptr}
	Suppose that $u(z)$ and $v(z)$ are two inner functions and there exists a M\"{o}bius transformation $\varphi(z)$ such that $v(z)=u(\varphi(z))$. Then, $\mathbb{U}\cong \mathbb{V}$.
\end{lemma}

\begin{proof}
	The sublevel set
	\begin{align*}
		\Omega_{u,\theta} = \{z\ ;\ |u(z)|<\theta\},
	\end{align*}
	\begin{align*}
		\Omega_{v,\theta}  &= \{z\ ;\ |v(z)|<\theta\}\\
		&=\{z\ ;\ |u(\varphi(z))|<\theta\}\\
		&=\{\varphi^{-1}(z)\ ;\ |u(\varphi(\varphi^{-1}(z)))|<\theta\}\\
		&=\{\varphi^{-1}(z)\ ;\ |u(z)|<\theta\}.
	\end{align*}
	$\varphi|_{\Omega_{v,\theta}}:\Omega_{v,\theta}
	\rightarrow\Omega_{u,\theta}$ is a homeomorphism which induce an isomorphism $\varphi_t:H_{0}(\Sigma_{v,t}; \mathbb{F})\rightarrow H_{0}(\Sigma_{u,t}; \mathbb{F})$, i.e., $V^{|v|}_{t}\cong V^{|u|}_{t}$. Let $\pi_{s,t}$ and $\widetilde{\pi}_{s,t}$ be the linear maps induced by the inclusions $i_{s,t}:\Omega_{v,\theta(s)} \to \Omega_{v,\theta(t)}$ and $\widetilde{i}_{s,t}:\Omega_{u,\theta(s)} \to \Omega_{u,\theta(t)}$, respectively. Then we have the following commutative diagram, 
	\begin{center}
		\begin{tikzcd}
			V_s^{|v|} \arrow{d}{\varphi_{s}} \arrow{r}{\pi_{s,t}} & V_t^{|v|} \arrow{d}{\varphi_{t}}\\
			V_s^{|u|} \arrow{r}{\widetilde{\pi}_{s,t}}         & V_t^{|u|}
		\end{tikzcd}
	\end{center}
	Then we have $\mathbb{U}\cong \mathbb{V}$. And also we have $d_{int}(\mathbb{U},\mathbb{V})=0$.
	
\end{proof}

\begin{lemma}\label{sur}
	Let $B(z)$ be a Blaschke product with Property $\mathfrak{B}$. Given any $\delta\in(0,1)$. Denote by $\mathbb{D}_{\delta}=\{w\in\mathbb{C};~|w|< \delta\}$. Let $\Omega_0$ be a component of $\Omega_{B,\delta} = \{z\in\mathbb{D}; |B(z)| < \delta\}$. Then, $B(z)\mid_{\overline{\Omega_{0}}}:\overline{\Omega_{0}}\rightarrow \overline{\mathbb{D}_{\delta}}$ is an analytic surjection. Moreover, for any $\zeta\in\partial\mathbb{D}_{\delta}$, there is an open neighborhood $V$ of $\zeta$ in the subspace $\overline{\mathbb{D}_{\delta}}$ such that $(B(z)\mid_{\overline{\Omega_{0}}})^{-1}(V)=\bigsqcup\limits_{j=1}^{m}U_j$, where each $U_j$ is an open subset in the subspace $\overline{\Omega_{0}}$ and  $B(z)\mid_{U_j}:U_j\rightarrow V$ is a bijection.
\end{lemma}

\begin{proof}
	Since $\overline{\Omega_{0}}\subseteq\mathbb{D}$, $B(z)$ is analytic on $\overline{\Omega_{0}}$.  $\Omega_{0}$ is a component of the level set $\Omega_{B,\delta}$, so $B(z)\mid_{\overline{\Omega_{0}}}:\overline{\Omega_{0}}\rightarrow \overline{\mathbb{D}_{\delta}}$ is an analytic surjection, $B(z)$ maps $\partial\Omega_{0}$ to $\partial{D_{\delta}}$,  and   consequently $B(z):\partial\Omega_{0}\rightarrow\partial{\mathbb{D}_{\delta}}$ is an analytic surjection. We write
	\[
	m=\deg(B\mid_{\overline{\Omega_{0}}})\triangleq \frac{1}{2\pi\mathbf{i}}\int_{\partial\Omega_{0}}\frac{d\xi}{\xi-\alpha}, \ \ \ \text{for any} \ \ \alpha\in\Omega_0.
	\]
	
	Given any $\zeta\in\partial\mathbb{D}_{\delta}$. Since $B(z)$ is an orientation-preserving map from the Jordan curve $\partial\Omega_{0}$ onto $\partial{D_{\delta}}$, there are exactly $m$ distinct points $\lambda_1, \cdots, \lambda_m$ on $\partial\Omega_{0}$ such that $B(\lambda_j)=\zeta$ for $j=1,\cdots, m$. Furthermore, we could choose sufficient small open neighborhood $V$ of $\zeta$ in the subspace $\overline{\mathbb{D}_{\delta}}$ such that $(B(z)\mid_{\overline{\Omega_{0}}})^{-1}(V)=\bigsqcup\limits_{j=1}^{m}U_j$, where each $U_j$ is an open neighborhood of $\lambda_j$ in the subspace $\overline{\Omega_{0}}$. Notice each point in $\overline{\mathbb{D}_{\delta}}$ has exactly $m$ preimages of $B(z)$ in $\overline{\Omega_{0}}$ (See Figure \ref{fig:degree_m}). Therefore, $B(z)\mid_{U_j}:U_j\rightarrow V$ is a bijection, for $j=1,\cdots, m$.
\end{proof}

\begin{figure}[h]
	\centering
	\includegraphics[width=1\textwidth]{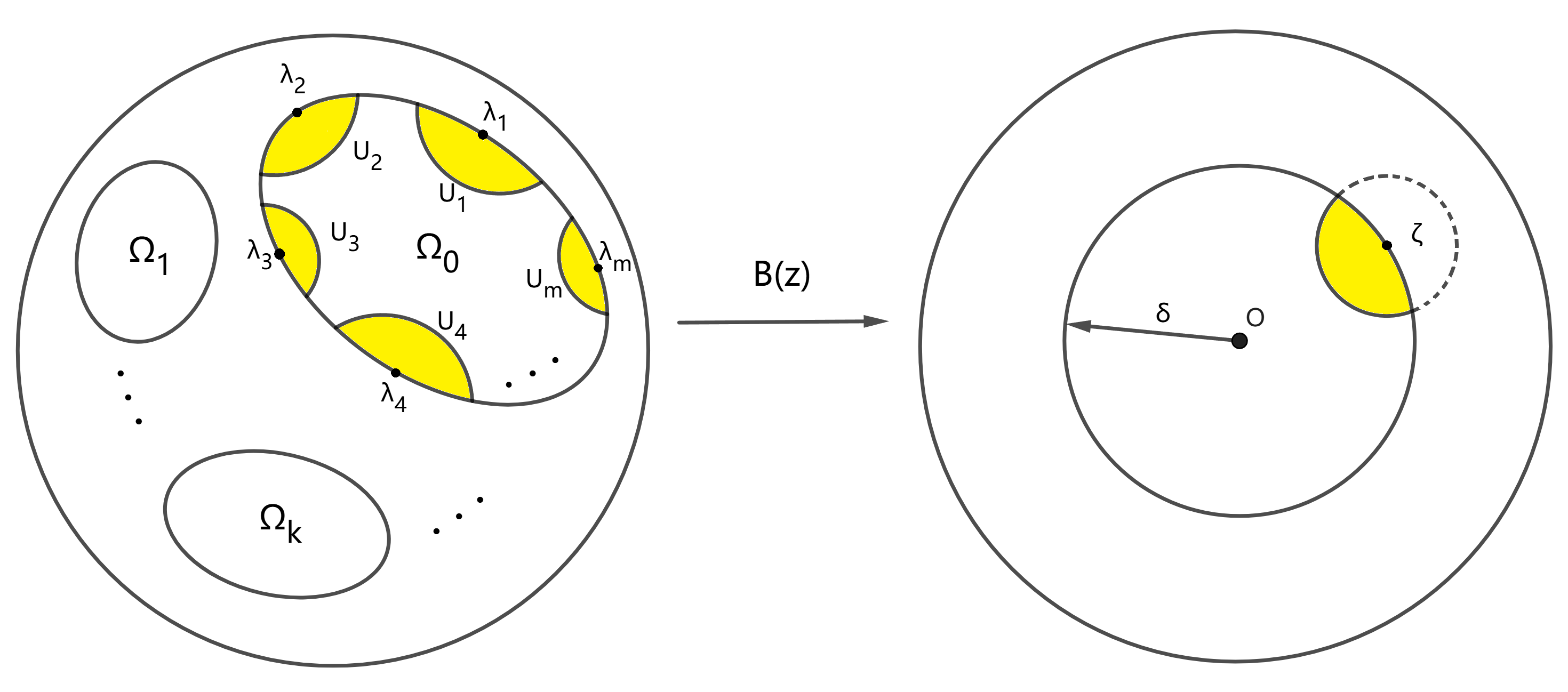}
	\caption{}
	\label{fig:degree_m}
\end{figure}

 The following lemma is a classical result about local behavior at critical points, see \cite[Theorem 11]{Ahlf-1978} for instance.

\begin{lemma}\label{zero}
	Let $f(z)$ be an analytic function at $z=z_{0}$, and $w_{0}=f(z_{0})$. Suppose that $f(z)-w_{0}$ has zero of order $n$ ($n\geq 1$) at $z_{0}$, i.e., $f'(z_{0})=f''(z_{0})=\cdots=f^{(n-1)}(z_{0})=0$ and $f^{(n)}(z_{0})\neq 0$. Then, for a sufficiently small positive number $\rho$, there exits a positive number $\mu$ such that when $0<|w-w_{0}|<\mu$, $f(z)-w$ has $n$ numbers of zeros of order $1$ in $0<|z-z_{0}|<\rho$.
\end{lemma}

Then we can get the relation between the order of a critical point and the number of components merging at the critical point. We say two domains $\Omega_1$ and $\Omega_2$ merge at $\omega$, if $\Omega_1\cap \Omega_2=\emptyset$ and $\{\omega\} \in \overline{\Omega_1}\cap \overline{\Omega_2}$.

\begin{theorem}\label{cri_and_com}
Suppose that $B(z)$ is a Blaschke product with Property $\mathfrak{B}$. Let $z_0\in\mathbb{D}$, then $z_0$ must be a common point of the boundaries of exactly $m$ (possibly infinitely many) distinct components of the level set $\{z\in\mathbb{D}; |B(z)| < |B(z_0)|\}$. When $z_0$ is a critical point of order $n-1$ for $B(z)$, i.e., $B^{(k)}(z_{0})=0$, $k=1,\ldots,n-1$, and $B^{(n)}(z_{0})\neq 0$, we have $n=m$.
\end{theorem}

\begin{proof}
	Consider $z_0$ and level set $\Omega_{B,|B(z_0)|}$. Then by Lemma \ref{zero}, for a sufficiently small positive number $\rho$, there exists a positive number $\mu$ such that when $0<|w-B(z_{0})|<\mu$, $B(z)-w$ has $n$ numbers of zeros of order $1$ in $0<|z-z_{0}|<\rho$. Let $\Omega_{1}, \Omega_{2}, \ldots, \Omega_{m}$ be the distinct components with $z_0$ as a common point in their boundaries. Then, by Lemma \ref{sur}, in a small neighborhood of $z_0$, there is only one preimage of $w$ in each $\Omega_{i}$ (See Figure \ref{fig:neibor}), $i=1,\cdots,m$, which implies $m=n$.

	\begin{figure}[h]
		\centering
		\includegraphics[width=1\textwidth]{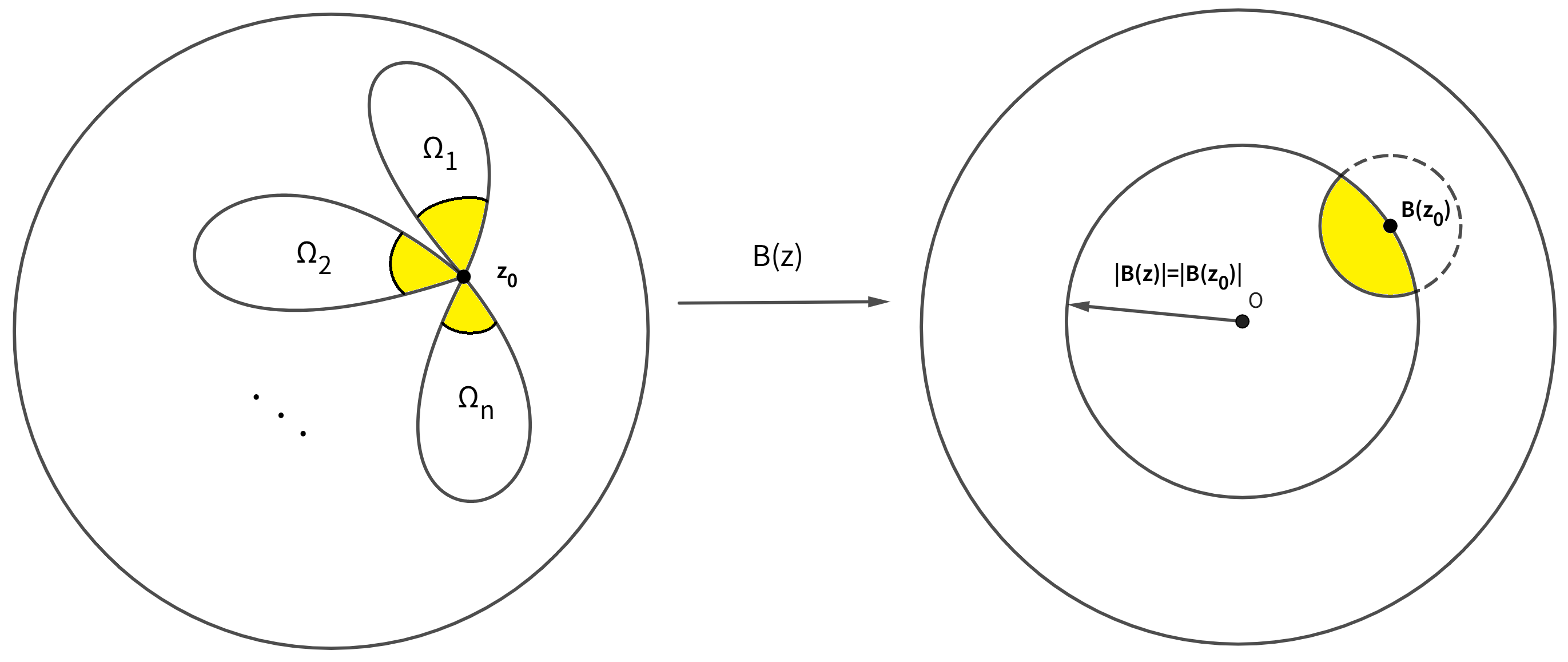}
		\caption{}
		\label{fig:neibor}
	\end{figure}
	
\end{proof}

Before we prove the Theorem A, we will give some definitions and the following six facts for Blashcke products $B$ with property $\mathfrak{B}$.

{\bf Fact 1} For each zero point $\alpha_k$, there exists $\theta_{k}\in (0,1)$ such that the component of $\Omega_{B,\theta_k}$ containing $\alpha_k$ has no other zero points. 
\begin{proof}
	Since $B$ is a Blaschke product with property $\mathfrak{B}$, we have the closure of $\Omega_{B,\theta_k}$ is in the unit disk. We claim  there  exist only finitely many zero points for each component of $\Omega_{B,\theta_k}$. Since otherwise, by the isolation of zero points of analytic functions, $B$ is identically zero which  conflicts with the definition of $B$. Let $\{D_1,\cdots,D_n; D_i\subset \Omega_0, i=1,\cdots,n\}$ be the set of disjoint open disks such that each of them contains only one zero point in $\Omega_0$ with $\alpha_k \in D_1$. By the minimal module principle, we have $\inf_{z\in \Omega_0\setminus{\cup^{n}_{i} D_i}}|B(z)|=\theta_0$ for some $\theta_{0}>0$ and $\inf_{z\in \partial D_1}|B(z)|\geq \theta_0$. For any $0<\theta_k<\theta_0$, we have the component of $\Omega_{B,\theta_k}$ containing $\alpha_k$ is in $D_1$. Since $D_1$ only contains $\alpha_k$, we can get the conclusion.
\end{proof}
{\bf Fact 2} A component of $\Omega_{B,\theta}$ can only merge with finitely many components of $\Omega_{B,\theta}$. Also, at any time, one component of $\Omega_{B,\theta}$ has at most finitely many zeros.
\begin{proof}
Suppose that a component $\Omega_0$ of $\Omega_{B,\theta}$ merge with infinitely many components $\Omega_1,\cdots,\Omega_n,\cdots$.   For $\theta'>\theta$, there is a component $\Omega\supseteq \cup_{j=0}^{+\infty}\Omega_j$ of $\Omega_{B,\theta'}$. By Lemma \ref{sc-components}, every component of $\Omega_{B,\theta'}$ contains at least one zero of the Blaschke product $B$, then the component $\Omega$  has infinitely many zero points, which contradicts with the isolation  of zero points of analytic functions. 
\end{proof}

 Let $A$ be  the ordered set of the zero points $\{\alpha_{k}\}$ of $B$ with $\alpha_i < \alpha_j$ when $i<j$. At time $t$,  the set of zero points contained in  one component $\Omega_t$ of $\Sigma_{B,t}$  has an order induced from $A$, and we call the smallest element among them $\alpha[\Omega_t]$. We also denote the component containing a zero point $\alpha$ by $\Omega_t(\alpha)$. 
 
{\bf Fact 3} At any time $t_0$, the component $\Omega_{t_0}$ containing $\alpha_k$ will finally merges with the component containing $\alpha_0$. Along the way, $\alpha[\Omega_t]$ changes finitely many times for $t>t_0$. 
\begin{proof}
There exists  $r<1$ such that $\alpha_k$ and $\alpha_0$ in  $U=\{z\in \mathbb{D};|z|<r\}$.  Let $\theta = \sup_{z\in \bar{U}}|B(z)|$,  by the maximum modulus principle, $|B(z)|< \theta$ for $z\in U$, we have $U\subset \Omega_{B,\theta}$. Since $U$ is connected, there must be a component of $\Omega_{B,\theta}$ containing $\alpha_0$ and $\alpha_k$.  Then we have the component $\Omega$ of $\Omega_{B,\theta}$ containing $\alpha_k$  also contains $\alpha_0$. By {\bf Fact 2}, $\Omega$ has only finitely many zero points, which implies $\alpha[\Omega_t]$ changes finitely many times.
\end{proof}

{\bf Fact 4} For two components  $\Omega_1$ and $\Omega_2$ of $\Omega_{B,
	\theta}$ merging at $\omega$, we have $\Omega_1\cap \Omega_2=\emptyset$ and $\{\omega\} = \overline{\Omega_1}\cap \overline{\Omega_2}$.
	\begin{proof}
	 Obviously, $\Omega_1 \cap \Omega_2=\emptyset$. Since the critical points are isolated, and each $\omega\in  \overline{\Omega_1}\cap \overline{\Omega_2}$ is a  critical point, so $\overline{\Omega_1}\cap \overline{\Omega_2}$ only contain finitely many points. Now we can choose two adjacent   critical points $\omega_1$, $\omega_2 \in\overline{\Omega_1}\cap \overline{\Omega_2}$. They separate each $\partial\Omega_i$ into two arcs, and we choose one arc from each $\partial\Omega_i$ such that they form a contour $\Gamma$ and the region $\Omega$ enclosed by it has no intersection with $\Omega_i$ (for example see Figure \ref{fig:2_compo_omega}). For any point $z_0\in \Omega$, there exists $\epsilon>0$ such that $|B(z_0)|>\theta+2\epsilon$. Also, $\sup_{z\in \Gamma}|B(z_0)|<\theta+\epsilon$. By the maximum modulus principle, $|B(z_0)|<\theta+\epsilon$ for $z_0\in\Omega$, which makes a contradiction.
	
\begin{figure}[h]
	\centering
	\includegraphics[width=0.8\textwidth]{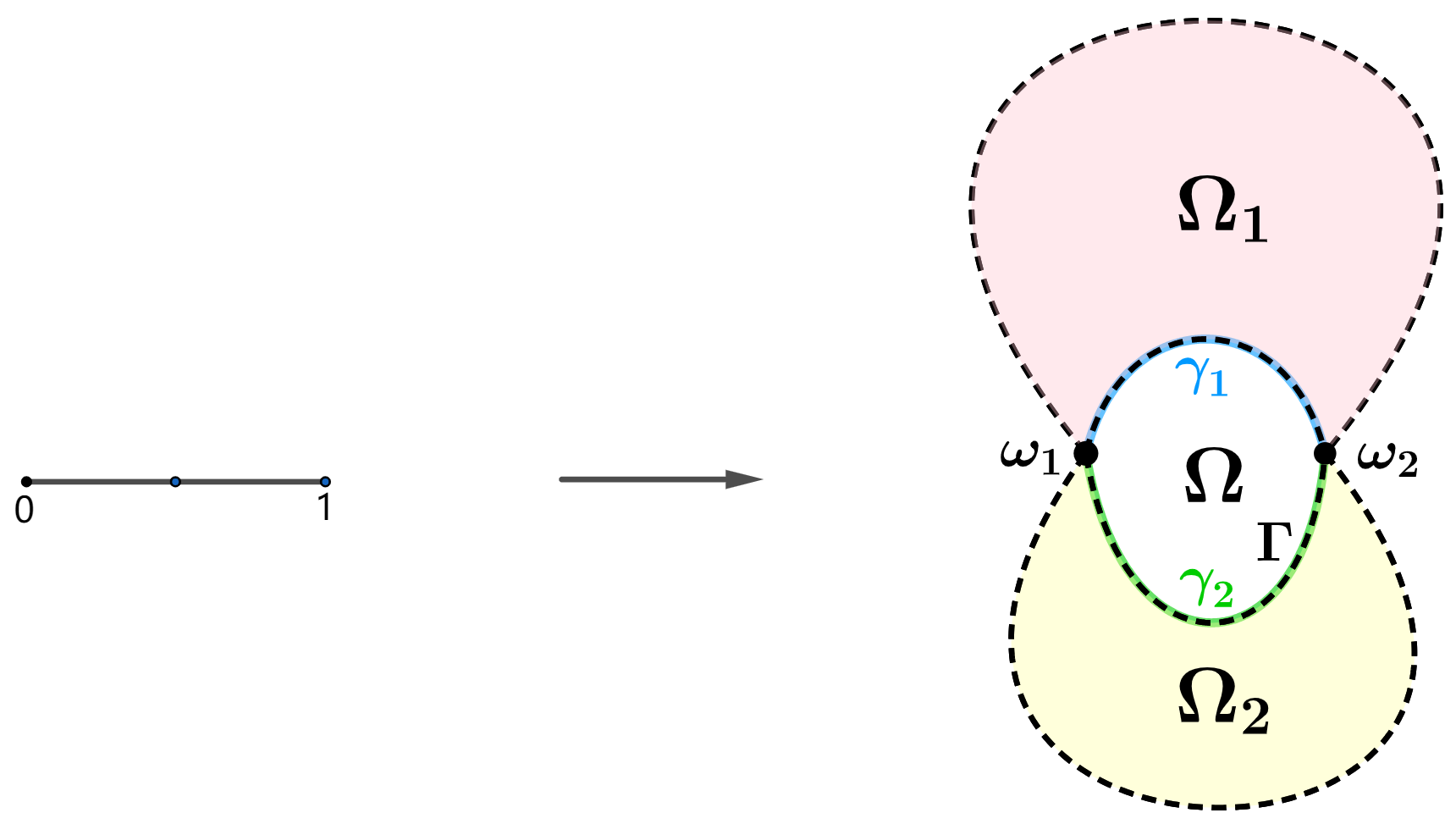}
	\caption{}
	\label{fig:2_compo_omega}
\end{figure}
\end{proof}

Let $W $ be the set of critical points $\{\omega_1,\omega_2,\cdots\}$. We define an equivalence relation on $W$ by $\omega_i\sim \omega_j$ if $|B(\omega_i)|=\theta=|B(\omega_j)|$ and the components sharing the critical point $\omega_i$ and those sharing the critical point $\omega_j$ are merging together right after  time $\theta$. We give an ordering of the equivalence classes and write the set of the equivalence classes $P=\{p_1, p_2, p_3,\cdots\}$. By {\bf Fact 2}, each $p_j=\{\omega_{j_1},\cdots,\omega_{j_{n_j}}\}$ is a finite set. Following the similar idea in {\bf Fact 4}, we can prove the following fact when more than two components merge together.

{\bf Fact 5} For each $p_j$, there do not exist simple closed curves $\Gamma$ (as shown in Figure \ref{fig:n_compo_omega}) on the boundaries of the merging components of $\Omega_{B,\theta}$ such that $\Gamma(0)=\omega_{j_{n_1}}$, $\Gamma(\frac{1}{n})=\omega_{j_{n_2}}$,$\cdots$, $\Gamma(\frac{n-1}{n})=\omega_{j_{n_{n}}}$ and $\Gamma(1)=\omega_{j_{n_1}}$ for any $\omega_{j_{n_k}}\in p_j$, $k=1,\cdots,n$.
\begin{figure}[h]
	\centering
	\includegraphics[width=1\textwidth]{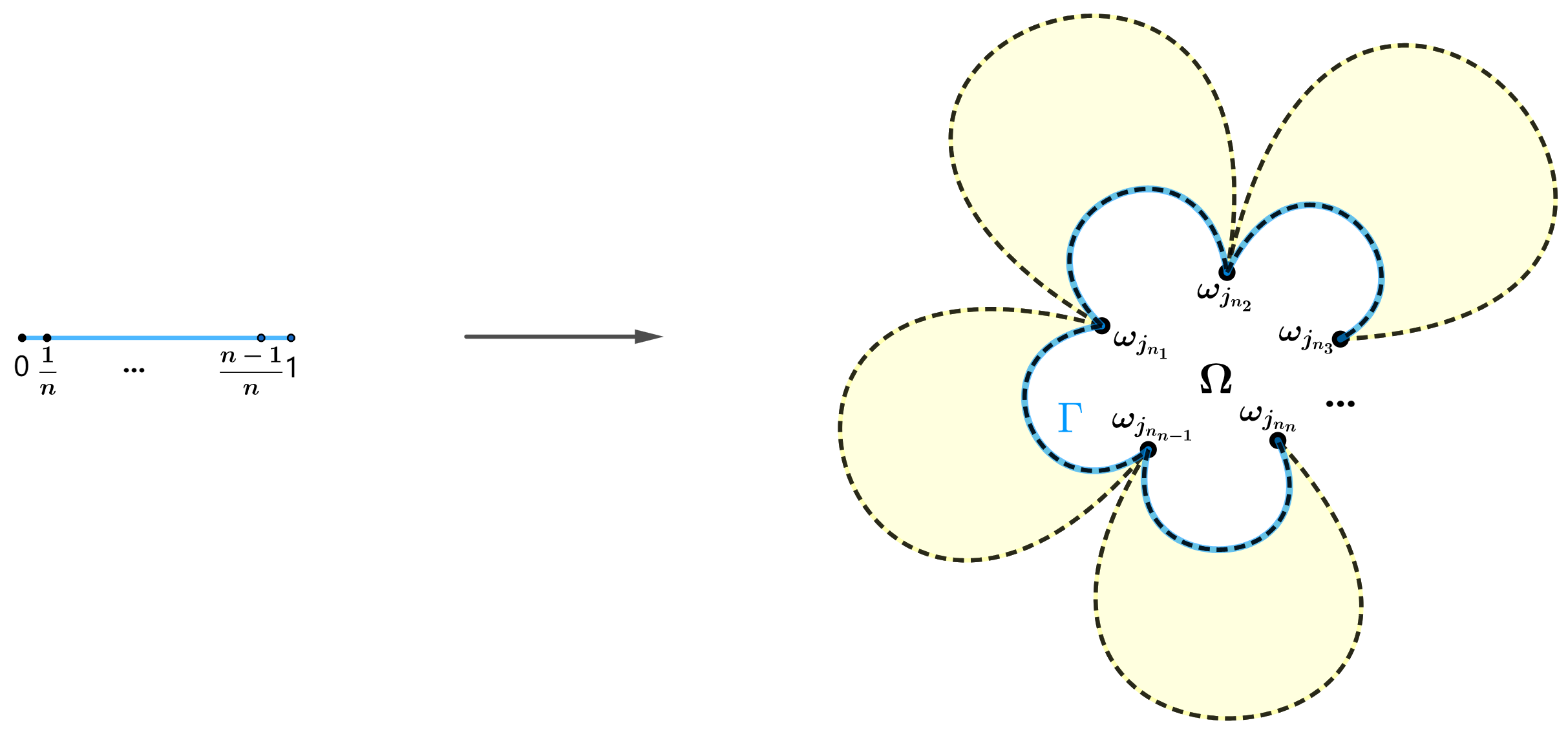}
	\caption{}
	\label{fig:n_compo_omega}
\end{figure}

{\bf Fact 6} Let $M_j+1$ be the number of the components  merging at the critical points of $p_j=\{\omega_{j_1},\cdots,\omega_{j_{n_j}}\}$ with order $m_{j_1},\cdots,m_{j_{n_j}}$ respectively. 
 We have $M_j=\sum_{k=1}^{n_j} m_{j_k}$, $k=1,\cdots,n_j$. 
\begin{proof}
	By  Theorem \ref{cri_and_com}, there are $m_{j_k}+1$ components merging at $\omega_{j_{k}}$. Since by {\bf Fact 5}, there exists no non-trivial loops  in the union of the closure of the components being merged. We have $M_j=\sum_{k=1}^{n_j} m_{j_k}$.
\end{proof}

Let $\mathbb{G}=(G, \widetilde{\pi})$ be the persistence module $\mathbb{F}(0,+\infty)\oplus(\oplus_{j}(\mathbb{F}(0,s_j])^{M_j})$, where $s_j=\ln\frac{1+|B(w_{j_1})|}{1-|B(w_{j_1})|}$ and $M_j=\sum_{k=1}^{n_j} m_{j_k}$. Let $(g_0)_t$ be the generator of $\mathbb{F}(0,+\infty)_t$. For each $p_j$ and let $(g_j^i)_t$ be the generator of  $\mathbb{F}(0,s_j]_t$, $i=1,\cdots,M_j$, where $\widetilde{\pi}_{s,t}((g_0)_s)=(g_0)_t$, $\widetilde{\pi}_{s,t}((g_j^i)_s)=(g_j^i)_t$.  Define $g_0:(0,+\infty)\rightarrow \mathbb{G}$ and $g_j^i(t):(0,+\infty)\rightarrow \mathbb{G}$ by
\[g_0(t)=(g_0)_t,~\text{for}~ t\in (0,+\infty)\]
\[g_j^i(t)=\left\{
	\begin{aligned}
		& (g_j^i)_t~&\text{for}~ t\in (0,s_j],\\
		& 0 ~&\text{for}~ t\notin (0,s_j].
	\end{aligned}
\right.
\]

Now, let us give the proof of Theorem A. Recall that $\mathbb{B}=(V^{|B|},\pi)$ is the persistence module induced by the Blashcke product $B$.  
\begin{proof}[\bf{Proof of Theorem A}]

	Let $B$ be a Blashcke product with property $\mathfrak{B}$. Consider $\Sigma_{B,t}=\{z;|B(z)|<\theta\}$, then $[\alpha_{k}]^{\mathbb{B}}_{t}$ is the generator of zeroth homology group of the component of $\Sigma_{B,t}$ containing $\alpha_{k}$.
	
	We have
	$$
	V^{|B|}_{t}=
	\text{span}_{\mathbb{F}}\{[\alpha_{k}]^{\mathbb{B}}_{t};\alpha_{k}\in A\}.
	$$


	By {\bf Fact 3}, for any $\alpha_k$ there exist $t_1<t_2<\cdots<t_{n_k}$ such that
	for each $t_i$, $\alpha[\Omega_{t}(\alpha_k)]$ will be updated\footnote{That is,  if $t_{i-1}< s_{i-1} \leq t_{i} <s_{i}\leq t_{i+1}$, $\alpha[\Omega_{s_{i-1}}(\alpha_k)]>\alpha[\Omega_{s_{i}}(\alpha_k)]$, where $i=1,\cdots,n_k$.}. Let $t_0=0$ and $t_{n_k+1}=+\infty$,  by {\bf Fact 1},  there exists a time period $(t_0,t_1]$ such that $\alpha_k = \alpha[\Omega_{s}(\alpha_k)]$ for $s\in (t_0,t_1]$. By {\bf Fact 3} $\alpha_0 = \alpha[\Omega_{s}(\alpha_k)]$  when $t_{n_k}<s<t_{n_k+1}$.


	
	For each $\alpha_k$ and time $t_1$ corresponding to $\alpha_k$, by Theorem \ref{cri_and_com} there must be a unique $\omega_{k'} \in W$ served as the common point on the boundaries of certain components, and one of which contains $\alpha_k$.  Define $b: A\rightarrow \mathbb{Z}^{+}$ to be the integer such that $\omega_{k'}\in p_{b(\alpha_k)}$.  For $p_{b(\alpha_k)} \in P$, by {\bf Fact 6}, there exist $M_{b(\alpha_k)}+1$ components merging at the critical points of $p_{b(\alpha_k)}$. Then we denote the merging components  by $\Omega_{t_1}^{0},\cdots,\Omega_{t_1}^{M_{b(\alpha_k)}}$ satisfying that  $\alpha[\Omega^{0}_{t_1}]<\cdots<\alpha[\Omega^{M_{b(\alpha_k)}}_{t_1}]$. 
	By the discussion above, there exists a unique  $\Omega_{t_1}^{l}$ containing $\alpha_k$, we define $a: A\rightarrow \mathbb{Z}^{+}$ by $a(\alpha_k)=l$. 
	
	Now we will define a map $F_{0}:A\rightarrow \{g_0\}\cup \{g^i_j; i=1,\cdots,M_j, j=1,\cdots ,|P| \}$ (Here when $|P|$ is not finite, we mean $j=1,2,\cdots$).
	 We  define $F_0(\alpha_0)=g_0$ and $F_0(\alpha_k)=g^{a(\alpha_k)}_{b(\alpha_k)}$.   To  illustrate this process, we provide Figure \ref{fig:cri_select} as an example. 
	\begin{figure}[h]
		\centering
		\includegraphics[width=0.7\textwidth]{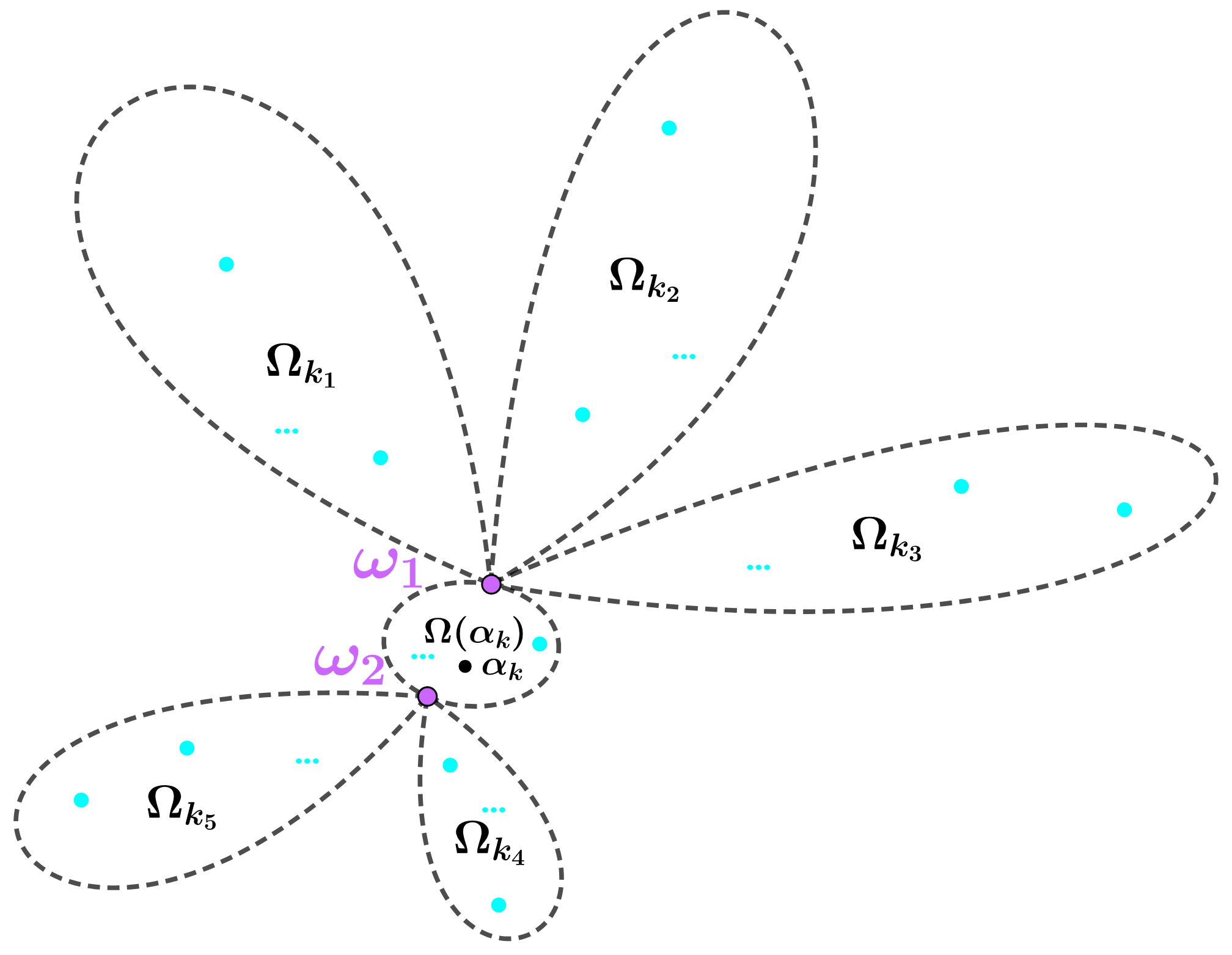}
		\caption{Five components $\Omega_{k_i}$ will merge with the component $\Omega(\alpha_{k})$ at $\omega_1$ and $\omega_2$, where $i=1,\cdots,5$. That is, there exists $p_j = \{\omega_1,\omega_2\}$. Each smallest zero point of the components satisfying that $\alpha[\Omega_{k_1}]>\alpha[\Omega_{k_2}]>\alpha_k>\alpha[\Omega_{k_4}]>\alpha[\Omega_{k_3}]>\alpha[\Omega_{k_5}]$. We have $F_0(\alpha[\Omega_{k_3}])=g_j^{1}$, $F_0(\alpha[\Omega_{k_4}])=g_j^{2}$, $F_0(\alpha_k)=g_j^{3}$, $F_0(\alpha[\Omega_{k_2}])=g_j^{4}$ and $F_0(\alpha[\Omega_{k_1}])=g_j^{5}$.}
		\label{fig:cri_select}
	\end{figure}
	
	Now we will show that $F_0$ is a bijection. If $F_0(\alpha_k)=F_0(\alpha_{k'})$ for two zero points $\alpha_k$ and $\alpha_{k'}$, we have $g^{a(\alpha_k)}_{b(\alpha_k)}=g^{a(\alpha_{k'})}_{b(\alpha_{k'})}$. That is, $a(\alpha_k)=a(\alpha_{k'})$ and $b(\alpha_k)=b(\alpha_{k'})$, which implies $\alpha_k=\alpha_{k'}$ since $a\times b$ is injective by construction. 
	For surjection, firstly we have $F_0(\alpha_0)=g_0$. For each $g_b^a$, we consider the smallest zero points $\alpha[\Omega_0]<\alpha[\Omega_1]<\cdots<\alpha[\Omega_{M_b}]$ in the components merging at the critical points in $p_b$, we have $F_0(\alpha[\Omega_i])=g_b^a$. 
	

	
	
	 For convenience, we denote the smallest zero points of $\Omega_{s}(\alpha_k)$ for $s\in (t_i,t_i+1]$ by $\alpha_{k(i)}$, where $i=0,\cdots,n_k-1$, and the smallest zero point of $\Omega_{s}(\alpha_k)$ for $s\in(t_{n_k},t_{n_k+1})$ by $\alpha_{k(n_k+1)}$, respectively, where $k(0)=k$ and $k(n_k+1)=0$.
	
	Next, define \begin{align*}
		F: &\mathbb{B}\rightarrow\mathbb{G}\\
		&[\alpha_0]^{\mathbb{B}}_t\mapsto(g_0)_t\\
		&[\alpha_k]^{\mathbb{B}}_t\mapsto\sum_{l=0}^{n_k+1}F_0(\alpha_{k(l)})(t).
	\end{align*}  Now we will show $F$ is an isomorphism.
	
	For injection, if $F([\alpha_k]^{\mathbb{B}}_t)=F([\alpha_{k'}]^{\mathbb{B}}_t)$, there must be integer numbers $m<n_k$ and $m'<n_{k'}$ such that
	\begin{align*}
		F([\alpha_k]^{\mathbb{B}}_t)=\sum_{l=0}^{m+1}F_0(\alpha_{k(l)})(t),\\
		F([\alpha_{k'}]^{\mathbb{B}}_t)=\sum_{l'=0}^{m'+1}F_0(\alpha_{k'(l')})(t),
	\end{align*}
 	where each element of $(F_0(\alpha_{k(l)}))_t$ and $(F_0(\alpha_{k'(l')}))_t$ is non-zero for $l=0,\cdots,m+1$ and $l'=0,\cdots,m'+1$. Since $F_0$ is a bijection,  we have $m=m'$. The fact that the smallest zero points satisfying that $k(l)>k(l+1)$ and $k'(l')>k'(l'+1)$ for $l=0,\cdots,m$ and $l'=0,
 	\cdots,m'$ implies
 	 $$F_0(\alpha_{k(l)})(t)=F_0(\alpha_{k'(l)})(t)\  \text{and}\  \alpha_{k(l)}=\alpha_{k'(l)}.$$
 	 In particularly, $\alpha_k=\alpha_{k(0)}=\alpha_{k'}$.

	 For each non-zero $(g^{i}_{j})_t \in G_t$, take $\alpha_k= (F_0)^{-1}(g^{i}_{j})$, we have $F([\alpha_k]^{\mathbb{B}}_t-[\alpha_{k(1)}]^{\mathbb{B}}_t)=(g^{i}_{j})_t$. Then $F$ is a surjection at each time $t$.
	 
	 Now we will prove that the diagram  commutes,
	 \begin{center}
	 	\begin{tikzcd}
	 		B_s \arrow{d}{F_{s}} \arrow{r}{\pi_{s,t}} & B_t\arrow{d}{F_t}\\
	 		G_s \arrow{r}{\tilde{\pi}_{s,t}}         & G_t
	 	\end{tikzcd}
	 \end{center}
	 For any sequence $\{s_1,s_2,\cdots,s_{n_k},s_{n_k+1}\}$ satisfying that $t_0<s_0\leq t_1<s_1\leq\cdots<s_{n_k}\leq t_{n_k}<s_{n_k}<t_{n_k+1}$, where $t_0=0$ and $t_{n_k+1}=+\infty$,
	 we have the following commutative diagram,
	 \begin{center}
	 	\begin{tikzcd}
	 		\left[ \alpha_{k}\right]^{\mathbb{B}}_{s_0}\arrow{r}{F_{s_0}} \arrow{d}{\pi_{{s_0},{s_1}}} &
	 		\sum_{l=0}^{n_k+1}F_0(\alpha_{k(l)})(s_0)  \arrow{d}{\tilde{\pi}_{s_0,s_1}} \\
	 		\left[ \alpha_{k}\right]^{\mathbb{B}}_{s_1} \arrow{r}{F_{s_1}} \arrow{d}{\vdots} & \sum_{l=1}^{n_k+1}F_0(\alpha_{k(l)})(s_1) \arrow{d}{\vdots}\\
	 		\left[\alpha_{k}\right]^{\mathbb{B}}_{s_{n_k}} \arrow{r}{F_{s_{n_k}}} \arrow{d}{\pi_{{s_{n_k}},{s_{n_k+1}}}} &\sum_{l=n_k}^{n_k+1}F_0(\alpha_{k(l)})(s_{n_k}) \arrow{d}{\tilde{\pi}_{s_{n_k},s_{n_k+1}}}\\
	 		\left[ \alpha_{k}\right]^{\mathbb{B}}_{s_{n_k+1}} \arrow{r}{F_{s_{n_k+1}}}& F_0(\alpha_{k(n_k+1)})(s_{n_k+1})	 		
	 	\end{tikzcd}
	 \end{center}
	 Notice that on the second level, the summation is starting from  $l=1$, because $F_0(\alpha_{k(0)})(s_1)=0$.
	 Then $\mathbb{B}\cong\mathbb{G}$. We complete the proof.
	 
\end{proof}

\begin{remark}	
	 By Theorem \ref{unique}, this decomposition is unique up to isomorphism.
\end{remark}

Noticing that there may be infinitely many $w_{j}$ with same function value $|B(w_j)|$,  which implies that the associated persistence modules are potentially not of locally finite type.  

\section{Persistence modules induced by inner functions with strong Property $\mathfrak{B}$ and interleaving distances}

In this section, the interleaving distances of the persistence modules induced by inner functions with strong Property $\mathfrak{B}$ are studied. We will prove Theorem B which is one of our main results. To prove it, we need a series of lemmas. Firstly, we show that every inner function with weak Property $\mathfrak{B}$ is a Blaschke product. The proof is  similar to the cases of Property $\mathfrak{B}$ \cite{Berman-1984} and Property $\mathfrak{H}$ \cite{Ne80}.

\begin{lemma}\label{weakB}
	Let $B$ be an inner function with weak Property $\mathfrak{B}$. Then, $B$ is a Blaschke product.
\end{lemma}
\begin{proof}
Notice that for every singular inner function $S$,  each component of the level set $\Omega_{S,\theta}$ (in fact, it has only one component) has boundary points on  $\partial{\mathbb{D}}$ for any $\theta\in(0,1)$, by Lemma \ref{sc-components}. Then, the inner function $B$ with weak Property $\mathfrak{B}$ can not have a singular inner function factor, and hence $B$ is a Blaschke product.
\end{proof}

The following lemma extends a result from Property $\mathfrak{H}$ to weak Property $\mathfrak{H}$, whose proof is essentially the same as the discussion of  Nestoridis \cite[p. 201]{Ne80}.
Recall that for an inner function $B$ with $\eta$-weak Property $\mathfrak{H}$, \[
\delta_{B,\eta}=\sup\{\rho(z_1, z_2);~z_1 \ \text{and} \ z_2 \ \text{belong to the same component of} \ \Omega_{B,\eta}\}.
\]
\begin{lemma}\label{zero number}
	Let $B$ be an inner function with $\eta$-weak Property $\mathfrak{H}$. Then each component of the level set $\Omega_{B,\eta}$ contains at least one and no more than ${\ln\eta}/{\ln\delta_{B,\eta}}$ numbers of zeros of $B$.
\end{lemma}

Now, let us continue to study the Blaschke products with weak Property $\mathfrak{H}$.

\begin{lemma}\label{rhoalphabeta}
Let $B$ be a Blaschke products with $\eta$-weak Property $\mathfrak{H}$. If $\widetilde{B}$ is an inner function satisfying $\|B-\widetilde{B}\|_{\infty}=\eta'<\eta$, then $\widetilde{B}$ has $(\eta-\eta')$-weak Property $\mathfrak{H}$. More precisely, we have 
\[
\Omega_{\widetilde{B},{\eta-\eta'}}\subseteq\Omega_{B,\eta} \ \ \ \ \ \text{and} \ \ \ \ \  \delta_{\widetilde{B}, \eta-\eta'}\le\delta_{B,\eta}.
\]  
Moreover, we may arrange the zero points of $B$ and $\widetilde{B}$ as $\{\alpha_{i}\}_{i}$ and $\{\beta_{i}\}_{i}$, respectively, such that each pair of $\alpha_i$ and $\beta_{i}$ are contained in the same component of the level set $\Omega_{B,\eta}$ and $\sup\rho(\alpha_i,\beta_i)<\delta_{B,\eta}$ for all $i$.
\end{lemma}
\begin{proof}
Consider the level set $\Omega_{B,\eta}=\{z\in\mathbb{D}; |B(z)|<\eta\}$. For any $z\in\mathbb{D}\setminus\Omega_{B,\eta}$, we have $|B(z)|\ge\eta$ and then together with $\|B-\widetilde{B}\|_{\infty}=\eta'<\eta$, one can see that
\begin{equation*}
	\left|\widetilde{B}(z)\right|=\left|B(z)-(B(z)-\widetilde{B}(z))\right|\ge\left||B(z)|-|B(z)-\widetilde{B}(z)|\right|\ge\eta-\eta'.
\end{equation*}
Consequently, $\Omega_{\widetilde{B},{\eta-\eta'}}\subseteq\Omega_{B,\eta}$, which implies each component of the level set $\Omega_{\widetilde{B}, \eta-\eta'}$ is contained in some component of $\Omega_{B,\eta}$. Then, 
\[
\delta_{\widetilde{B}, \eta-\eta'}\le\delta_{B,\eta},
\]
and hence $\widetilde{B}$ has $(\eta-\eta')$-weak Property $\mathfrak{H}$.

Let $\Omega_0$ be an arbitrary component of the level set $\Omega_{B,\eta}$. For any $\xi\in\partial\Omega_0\subseteq\mathbb{D}$, we have
\begin{equation*}
	|B(\xi)|=\eta>\|B-\widetilde{B}\|_{\infty}\ge|B(\xi)-\widetilde{B}(\xi)|.
\end{equation*}
By Rouch\'{e}'s theorem, the analytic functions $B$ and $\widetilde{B}$ have the same number of zeros in the component  $\Omega_0$. Then, we may arrange the zero points of $B$ and $\widetilde{B}$ as $\{\alpha_{i}\}_{i}$ and $\{\beta_{i}\}_{i}$ respectively, such that each pair of $\alpha_i$ and $\beta_{i}$ are contained in the same component of the level set $\Omega_{B,\eta}$. Furthermore, one can see that for every $i$,
\[
\sup\rho(\alpha_i,\beta_i)<\delta_{B,\eta}.
\]
\end{proof}

\begin{corollary}
The set of the inner functions with weak Property $\mathfrak{H}$ is open in the space of all inner functions.
\end{corollary}

Let $B$ be a Blaschke products with $\theta$-weak Property $\mathfrak{B}$ and let $\widetilde{B}$ be an inner function satisfying $\|B-\widetilde{B}\|_{\infty}=\theta'<\theta$. Similar to the poof of  Lemma \ref{rhoalphabeta}, one can also see that the level set $\Omega_{\widetilde{B},{\theta-\theta'}}$ is contained in the level set $\Omega_{B,\theta}$ and consequently $\widetilde{B}$ has $(\theta-\theta')$-weak Property $\mathfrak{B}$. Then, we obtain the following result immediately.

\begin{corollary}
The set of the inner functions with weak Property $\mathfrak{B}$ is open in the space of all inner functions.
\end{corollary}

\begin{lemma}\label{limto0}
	Let $B$ be a Blaschke product with weak Property $\mathfrak{H}$. Then 
	\begin{equation}
		\lim\limits_{\eta\to0}\delta_{B,\eta}=0.
	\end{equation}
\end{lemma}
\begin{proof}
	Let 
	\[
	B(z)=\lambda\prod\limits_{n}\frac{|\alpha_{n}|}{\alpha_{n}}\left(\frac{\alpha_{n}-z}{1-\overline{\alpha_{n}}z}\right), \ \ \ \ \ \text{where} \ |\lambda|=1,
	\] 
	be an infinite or finite Blaschke product with $\eta$-weak Property $\mathfrak{H}$. Given an arbitrary  component $\Omega$ of the level set $\Omega_{B,\eta}$. Denote by $N$ the number of zeros in the component $\Omega$. By Lemma \ref{zero number}, 
	\[
	1\le N\le\frac{\ln\eta}{\ln\delta_{B,\eta}}.
	\]
	Without generality, let $\{\alpha_n\}_{n=1}^{N}$ be the zeros of $B$ lying in the component $\Omega$. Furthermore, let
	\[
	\widehat{B}(z)=\lambda\prod\limits_{n\ge N+1}\frac{|\alpha_{n}|}{\alpha_{n}}\left(\frac{\alpha_{n}-z}{1-\overline{\alpha_{n}}z}\right).
	\]
Then,
	\[
		B(z)=\left(\prod\limits_{n=1}^{N}\frac{|\alpha_{n}|}{\alpha_{n}}\left(\frac{\alpha_{n}-z}{1-\overline{\alpha_{n}}z}\right)\right)\cdot\widehat{B}(z).
	\]
For any $\xi\in\partial\Omega$, we have 
\[
	\eta=\left|B(\xi)\right|=\left(\prod\limits_{n=1}^{N}\rho(\alpha_n,\xi)\right)\cdot\left|\widehat{B}(\xi)\right|
	\le\left(\delta_{B,\eta}\right)^{N}\cdot\left|\widehat{B}(\xi)\right|,
\]
and consequently,
\begin{equation}
\left|\widehat{B}(\xi)\right|\ge\frac{\eta}{\left(\delta_{B,\eta}\right)^{N}}.
\end{equation}

Since the analytic function $\widehat{B}$ has no zeros in the set $\Omega$, it follows from the minimal module principle that
\begin{equation*}
	\inf\limits_{z\in\Omega}\left|\widehat{B}(z)\right|\ge\frac{\eta}{\left(\delta_{B,\eta}\right)^{N}}.
\end{equation*}

For any $K>\frac{1}{\eta}$, let 
\[
\theta=\frac{\delta_{B,\eta}}{K\cdot\frac{\ln\eta}{\ln\delta_{B,\eta}}}.
\]
it is obvious that 
\[\theta\le\frac{\delta_{B,\eta}}{KN}<\eta. 
\]
Furthermore, let
\begin{equation*}
	\Omega'=\{z\in\mathbb{D};\  \text{there}\ \text{is}\ \alpha_n, \ 1\le n\le N,\ \text{such}\ \text{that}\ \rho(z,\alpha_{n})<\theta\}=\bigcup\limits_{n=1}^{N}D_{\rho}(\alpha_{n}, \theta).
\end{equation*}
It is not difficult to see that $\Omega'\subseteq\Omega$. Denote
\[
\widetilde{\eta}=\frac{\eta}{K^{\ln\eta/\ln\delta_{B,\eta}}\left(\frac{\ln\eta}{\ln\delta_{B,\eta}}\right)^{\ln\eta/\ln\delta_{B,\eta}}}.
\]
Then, for any $z\in\Omega\textbackslash\Omega'$,

\begin{align*}
	\left|B(z)\right|&=~\left(\prod\limits_{n=1}^{N}\rho(\alpha_n,z)\right)\cdot\left|\widehat{B}(z)\right|
	\\
	&~\ge \theta^{N}\cdot\frac{\eta}{\left(\delta_{B,\eta}\right)^{N}}\\
	&\ge~\frac{\delta_{B,\eta}^{N}}{K^{N}\left(\frac{\ln\eta}{\ln\delta_{B,\eta}}\right)^{N}}\cdot\frac{\eta}{\left(\delta_{B,\eta}\right)^{N}}\\
	&=~\frac{\eta}{K^{N}\left(\frac{\ln\eta}{\ln\delta_{B,\eta}}\right)^{N}} \\
	&\ge~\widetilde{\eta}.
\end{align*}
Consequently, each component of the level set $\Omega_{B,\widetilde{\eta}}$ contained in $\Omega'$ must be contained in one component of $\Omega'=\bigcup\limits_{n=1}^{N}D_{\rho}(\alpha_{n}, \theta)$. Denote
	\[
\delta'=\sup\{\rho(z_1, z_2);~z_1 \ \text{and} \ z_2 \ \text{belong to the same component of} \ \Omega'\}.  
   \]
Although $\Omega'$ may have several components (see Figure \ref{fig:con_51} for example), a component of $\Omega'$ is the union of at most $N$ pseudo-hyperbolic disks $D_{\rho}(\alpha_{n}, \theta)$.
\begin{figure}[htbp]
	\centering
	\includegraphics[width=0.8\textwidth]{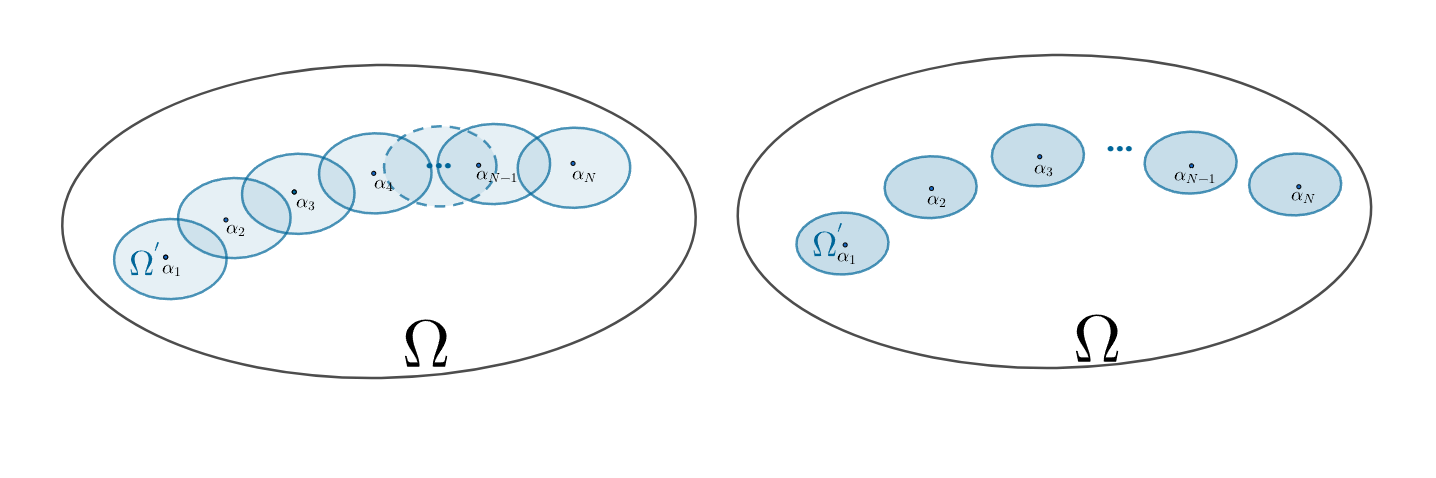}
	\caption{Components of $\Omega'$}
	\label{fig:con_51}
\end{figure} 
We always have
\begin{equation*}
	\delta'\le2N\theta\le2N\cdot\frac{\delta_{B,\eta}}{KN}=\frac{2\delta_{B,\eta}}{K}.
\end{equation*}
Then, we obtain that
\[
\delta_{B,\widetilde{\eta}}\le\delta'\le\frac{2\delta_{B,\eta}}{K}.
\]
It is easy to see that as $K\rightarrow+\infty$, $\widetilde{\eta}\rightarrow0$ and $\delta_{B,\widetilde{\eta}}\rightarrow0$. This completes the proof.
\end{proof}

\begin{lemma}\label{ineq3}
	Let $x\in(0,1)$ and $y>1$. Then,
	\begin{equation}
		\frac{1-x}{1-x^{1/y}}<y.
	\end{equation} 
	Moreover,
	\begin{equation}
		\frac{1+x^{1/y}}{1-x^{1/y}}\cdot\frac{1-x}{1+x}< y^2.
	\end{equation}
\end{lemma}
\begin{proof}
	Consider the function $f(x)=1-x-y(1-x^{1/y})$ with $x$ as the variable on $(0,1)$. We have
	\[
	f'(x)=-1+y\cdot\frac{1}{y}\cdot x^{1/y-1}=x^{1/y-1}-1.
	\]
	It is easy to see that $f'(x)>0$ for any $x\in(0,1)$, and consequently $f(x)$ is a monotonically increasing function on $(0,1)$. Then for any $x\in(0,1)$, $f(x)<f(1)=0$ and hence 
	\begin{equation}\label{inequ}
		\frac{1-x}{1-x^{1/y}}< y.
	\end{equation}
	Furthermore, 
	\begin{align*}
		\frac{1+x^{1/y}}{1-x^{1/y}}\cdot\frac{1-x}{1+x}
		=&~\left(\frac{1+x^{1/y}}{1-x^{1/y}}\cdot\frac{1-x^{1/y}}{1-x^{1/y}}\right)\left(\cdot\frac{1-x}{1+x}\cdot\frac{1-x}{1-x}\right)\\
		=&~\frac{1-x^{2/ y}}{\left(1-x^{1/y}\right)^{2}}\cdot\frac{(1-x)^{2}}{1-x^2}\\
		=&~\frac{1-x^{2/ y}}{1-x^2}\cdot\left(\frac{1-x}{1-x^{1/y}}\right)^{2}.
	\end{align*}
	Since $x\in(0,1)$ and $y>1$, it is obvious that $\frac{1-x^{2/y}}{1-x^2}<1$. Then, by inequality (\ref{inequ}), we have
	\begin{align*}
		\frac{1+x^{1/y}}{1-x^{1/y}}\cdot\frac{1-x}{1+x}
		=&~\frac{1-x^{2/y}}{1-x^2}\cdot\left(\frac{1-x}{1-x^{1/y}}\right)^{2} \\
		<&~\left(\frac{1-x}{1-x^{1/y}}\right)^{2} \\
		<&~y^2.
	\end{align*}
		This completes the proof.
	\end{proof}

\begin{lemma}\label{ln1x}
	For any $x\in(0,1)$, we have $\ln(1-x)>\frac{-x}{1-x}$.
\end{lemma}
\begin{proof}
	Let $f(x)=\ln(1-x)+\frac{x}{1-x}$ for  $x\in(0,1)$. Then,
	\begin{equation*}
		f'(x)=\frac{x}{(1-x)^{2}}.
	\end{equation*}
	It is easy to see that $f'(x)>0$ for any $x\in(0,1)$, and consequently the function $f(x)$ is a monotonically increasing function on $(0,1)$. Then for any $x\in(0,1)$, $f(x)>f(0)=0$ and hence 
	\[
	\ln(1-x)>\frac{-x}{1-x}.
	\]
\end{proof}

\begin{lemma}\label{ineq2}
	Given any $\delta_0\in(0,1/\mathrm{e})$. Let 
	\[
	\gamma=1+2\frac{1}{\sqrt{-\ln\delta_0}} \ \ \ \ \text{and} \ \ \ \ \delta=\delta_0 ^{1-{1}/{\sqrt{-\ln\delta_0}}}. 
	\]
	If $\rho(x,y)\le\delta_0$ and $\rho(x,z)\ge\delta$, then  $\rho(y,z)\ge\rho(x,z)^{\gamma}$.
\end{lemma}
\begin{proof}
	Firstly, we have
	\begin{align*}
		\rho(y,z)&\ge\frac{\rho(x,z)-\rho(x,y)}{1-\rho(x,z)\rho(x,y)}\\
		&\ge\frac{\rho(x,z)-\delta_0}{1-\delta_0 \rho(x,z)}.
	\end{align*}
	
	Since $\rho(x,z)\ge\delta=\delta_0 ^{1-{1}/{\sqrt{-\ln\delta_0}}}$, there is $\alpha\in(0, 1-{1}/{\sqrt{-\ln\delta_0}})\subseteq(0,1)$ such that $\rho(x,z)=\delta_0 ^{\alpha}$. Then, by Lemma \ref{ineq3}, we have
	\begin{align*}
		\frac{\rho(x,z)-\delta_0}{1-\delta_0 \rho(x,z)}=&~\frac{\delta_0 ^{\alpha}-\delta_0}{1-\delta_0 \delta_0 ^{\alpha}}\\
		=&~\delta_0 ^{\alpha}\cdot\frac{1-\delta_0 ^{1-\alpha}}{1-\delta_0 ^{1+\alpha}}\\
		=&~\delta_0 ^{\alpha}\cdot\frac{1}{\frac{1-\delta_0 ^{1+\alpha}}{1-\left(\delta_0 ^{1+\alpha}\right)^{(1-\alpha)/(1+\alpha)}}}\\
		\ge&~\delta_0 ^{\alpha}\cdot\frac{1-\alpha}{1+\alpha}.
	\end{align*}
	
	Now, we want to show for  $\alpha\in(0, 1-{1}/{\sqrt{-\ln\delta_0}})\subseteq(0,1)$, 
	\begin{equation*}
		\delta_0 ^{\alpha}\cdot\frac{1-\alpha}{1+\alpha}\ge\delta_0 ^{\alpha\gamma}.
	\end{equation*}
	Since $\alpha\le1-\frac{1}{\sqrt{-\ln\delta_0}}$, we have $\frac{1}{1-\alpha}\le\sqrt{-\ln\delta_0}$. Then, by Lemma \ref{ln1x}, we have
	\begin{align*}
		\ln\left(\frac{1-\alpha}{1+\alpha}\right)&=~\ln\left(1-\frac{2\alpha}{1+\alpha}\right)\\
		&\ge ~\frac{-\frac{2\alpha}{1+\alpha}}{1-\frac{2\alpha}{1+\alpha}}\\
		&=~\frac{-2\alpha}{1-\alpha}\\
		&\ge~-2\alpha\sqrt{-\ln\delta_0}\\
		&=~\frac{2\alpha}{\sqrt{-\ln\delta_0}}\cdot\ln\delta_0.
	\end{align*}
	
	Consequently, one can see that 
	\[
	\frac{1-\alpha}{1+\alpha}\ge\delta_0 ^{{2\alpha}/{\sqrt{-\ln\delta_0}}}=\delta_0 ^{\alpha(\gamma-1)}.
	\]
	Therefore,
	\begin{align*}
		\rho(y,z)&\ge~\frac{\rho(x,z)-\delta_0}{1-\delta_0 \rho(x,z)}\\
		&\ge~\delta_0 ^{\alpha}\cdot\frac{1-\alpha}{1+\alpha}\\
		&\ge~\delta_0 ^{\alpha}\cdot\delta_0 ^{\alpha(\gamma-1)}   \\
		&=~\delta_0 ^{\alpha\gamma}\\
		&=~\rho(x,z)^{\gamma}.
	\end{align*}
This finishes the proof.
\end{proof}

\begin{remark}
	Motivated by  Earl \cite{JP1970},  Nestoridis proposed a lemma  in \cite{Ne80}: "Let $\gamma>1$ (Nestoridis use the notation $\epsilon$ in his paper) and $\delta_0 <\frac{\gamma-1}{\gamma+1}$. Then there exists $\delta<1$ such that the inequalities $\rho(x,y)\le\delta_0$ and $\rho(x,z)\ge\delta$ imply $\rho(y,z)\ge\rho(x,z)^{\gamma}$." However, neither  Earl nor  Nestoridis provided a specific explanation of the relationship between $\delta$ and $\delta_0$. Inspired by their work, we propose Lemma \ref{ineq2}. We mainly want to demonstrate that $\delta_0 \to 0$ implies $\delta\to0$ and $\gamma\to1^{+}$, which is useful to prove Theorem B.
\end{remark}

Review that each  Blaschke product $B$ with strong Property $\mathfrak{B}$ induces a persistence module $\mathbb{B}=(V^{|B|},~\pi)=((V^{|B|})_0,~\pi)$. Next, we will demonstrate the relation between the supremum norm of Blaschke products with strong Property $\mathfrak{B}$ and the interleaving distance of the induced persistence modules, which plays an important role in the proof of Theorem B.

\begin{lemma}\label{dintDelta}
	Let $B$ and $\widetilde{B}$ be two Blaschke products with $\eta$-strong Property $\mathfrak{B}$. Suppose that $\delta_0=\min\{\delta_{B,\eta},\delta_{\widetilde{B},\eta}\}<1/\mathrm{e}$ and $\delta=\delta_{0} ^{1-{1}/{\sqrt{-\ln\delta_{0}}}}<1-2\eta$. Let 
	\[
	\gamma=1+\frac{2}{\sqrt{-\ln\delta_{0}}} \ \ \ \text{and} \ \ \ \ 
	T_0=\ln\frac{1+(\delta+2\eta)}{1-(\delta+2\eta)}.
	\]
	Then, $\|\widetilde{B}-B\|_{\infty}<\eta$ implies
	\begin{equation*}
		d_{int}(\mathbb{B},\widetilde{\mathbb{B}})\le\Delta=\max\left\{T_0, ~2\ln\gamma\right\}.
	\end{equation*} 
\end{lemma}

\begin{proof}
	Firstly, we claim three results as follows.
	
	\vspace{2mm}
	
	\textbf{Claim 1.} We could arrange the zero points of $B$ and $\widetilde{B}$ as $\{\alpha_{i}\}_{i}$ and $\{\beta_{i}\}_{i}$ respectively, such that $\sup\limits_{i\in \Lambda}\rho(\alpha_i,\beta_i)<\delta_0$ and for any $t\ge\ln\frac{1+\delta}{1-\delta}$, 
	each pair of $\alpha_i$ and $\beta_{i}$ are contained in the same component of the level set $\Sigma_{B, t}$ and in the same component of the level set $\Sigma_{\widetilde{B},t}$.
	
	\vspace{3mm}
	
	Denote by  $\{\alpha_{i}\}_{i\in\Lambda}$ and $\{\beta_{i}\}_{i\in\Lambda}$ the zero sets of $B$ and $\widetilde{B}$, where the index set $\Lambda$ is finite or countably  infinite. Since $B$ is a Blaschke product with $\eta$-strong Property $\mathfrak{B}$ and $\|\widetilde{B}-B\|_{\infty}<\eta$, it follows from Lemma \ref{rhoalphabeta} that there is a bijection between $\sigma: \Lambda\rightarrow\Lambda$, such that $\sup\limits_{i\in\Lambda}\rho(\alpha_i,\beta_{\sigma(i)})<\delta_{B,\eta}$ and each pair of $\alpha_i$ and $\beta_{\sigma(i)}$ are contained in the same component of the level set $\Omega_{B,\eta}$. 
	
	Given any $\theta\ge\delta_{B,\eta} ^{1-{1}/{\sqrt{-\ln\delta_{B,\eta}}}}$. It is worthy  noticing that,  by Lemma \ref{zero number},
	\[
	\theta\ge\delta_{B,\eta} ^{1-{1}/{\sqrt{-\ln\delta_{B,\eta}}}}>\delta_{B,\eta}\ge\eta.
	\] 
	Then, each pair of $\alpha_i$ and $\beta_{\sigma(i)}$ are also contained in the same component of the level set $\Omega_{B,\theta}$. 
	
	Let $\widetilde{\Omega}$ be an arbitrary component of the level set $\Omega_{\widetilde{B},\theta}$. Suppose that the zero point $\beta_{\sigma(i)}$ belongs to $\widetilde{\Omega}$.  For any $\xi\in \partial\widetilde{\Omega}$, we have
	\begin{equation*}
		\rho(\beta_{\sigma(i)},\xi)\ge|\widetilde{B}(\xi)|=\theta\ge\delta_{B,\eta} ^{1-{1}/{\sqrt{-\ln\delta_{B,\eta}}}}>\delta_{B,\eta}>\rho(\alpha_i,\beta_{\sigma(i)}).
	\end{equation*}
	Since $\widetilde{\Omega}$ is a simple domain, $\alpha_i$ also belongs to $\widetilde{\Omega}$. Thus, for any $\theta\ge\delta_{B,\eta} ^{1-{1}/{\sqrt{-\ln\delta_{B,\eta}}}}$, each pair of $\alpha_i$ and $\beta_{\sigma(i)}$ are contained in the same component of the level set $\Omega_{B, \theta}$ and in the same component of the level set $\Omega_{\widetilde{B},\theta}$. That means, for any 
	\[
	t\ge\ln\frac{1+\delta_{B,\eta} ^{1-{1}/{\sqrt{-\ln\delta_{B,\eta}}}}}{1-\delta_{B,\eta} ^{1-{1}/{\sqrt{-\ln\delta_{B,\eta}}}}},
	\]
	each pair of $\alpha_i$ and $\beta_{\sigma(i)}$ are contained in the same component of the level set $\Sigma_{B, t}$ and in the same component of the level set $\Sigma_{\widetilde{B},t}$.

    Similarly, since $\widetilde{B}$ is a Blaschke product with $\eta$-strong Property $\mathfrak{B}$ and $\|\widetilde{B}-B\|_{\infty}<\eta$, there is a bijection  $\widetilde{\sigma}: \Lambda\rightarrow\Lambda$, such that $\sup\limits_{i\in\Lambda}\rho(\alpha_i,\beta_{\widetilde{\sigma}(i)})<\delta_{\widetilde{B},\eta}$ and for any 
    \[
    t\ge\ln\frac{1+\delta_{\widetilde{B},\eta} ^{1-{1}/{\sqrt{-\ln\delta_{\widetilde{B},\eta}}}}}{1-\delta_{\widetilde{B},\eta} ^{1-{1}/{\sqrt{-\ln\delta_{\widetilde{B},\eta}}}}},
    \]
    each pair of $\alpha_i$ and $\beta_{\widetilde{\sigma}(i)}$ are contained in the same component of the level sets $\Sigma_{\widetilde{B},t}$ and in the same component of the level set $\Sigma_{B, t}$.

    If $\delta_{B,\eta}\le\delta_{\widetilde{B},\eta}$,  we rearrange the zero points of $\widetilde{B}$ so that the $i$-th zero point is $\beta_{\sigma(i)}$, we also call it $\beta_i$; and if $\delta_{B,\eta}>\delta_{\widetilde{B},\eta}$, we rearrange the zero points of $\widetilde{B}$ so that the $i$-th zero point is $\beta_{\widetilde{\sigma}(i)}$, we also call it $\beta_i$. Then, 
    \[
    \sup\limits_{i\in \Lambda}\rho(\alpha_i,\beta_i)<\min\{\delta_{B,\eta},\delta_{\widetilde{B},\eta}\}=\delta_0,
    \]
    and for any  $t\ge\ln\frac{1+\delta}{1-\delta}$, 
    each pair of $\alpha_i$ and $\beta_{i}$ are contained in the same component of the level set $\Sigma_{B, t}$ and in the same component of the level set $\Sigma_{\widetilde{B},t}$. This proves the Claim 1.
    
	\vspace{2mm}
	
	From now on, we denote the zero points of $B$ and $\widetilde{B}$ by $\{\alpha_{i}\}_{i\in\Lambda}$ and $\{\beta_{i}\}_{i\in\Lambda}$ arranged as stated in Claim 1.
	  
	\vspace{2mm}
	
	\textbf{Claim 2.}   Given any  $t\ge\ln\frac{1+(\delta+\eta)}{1-(\delta+\eta)}$.
		\begin{enumerate}
		\item[(1)] If $\alpha_i$ and $\alpha_j$ are contained in the same component of the level set $\Sigma_{B,t}$, then $\beta_{i}$ and $\beta_j$ are contained in the same component of the level set $\Sigma_{\widetilde{B},t+2\ln\gamma}$.
		\item[(2)] 	 If $\beta_m$ and $\beta_n$ are contained in the same component of the level set $\Sigma_{\widetilde{B},t}$, then $\alpha_m$ and $\alpha_n$ are contained in the same component of the level set $\Sigma_{B,t+2\ln\gamma}$.
	\end{enumerate}
	
	\vspace{2mm}
	
	Denote
	\[
	\theta=\theta(t)=\frac{\mathrm{e}^t-1}{\mathrm{e}^t+1}.  \ \ \ \  \ \left(\text{i.e.}, \ \ t=\ln\frac{1+\theta}{1-\theta} \right)
	\]
	Then, $\theta\in[\delta+\eta, 1)$. Let $\Omega$ be an arbitrary component of the level set $\Omega_{B,\theta}$. Notice that 
	\[
	\gamma=\gamma(\eta)=1+\frac{2}{\sqrt{-\ln\delta_{0}}}
	\]
	and for any $\xi\in\partial\Omega$ and any $k\in\Lambda$,
	\[
	\rho(\alpha_k,\beta_k)<\delta_0 \ \ \ \ \text{and} \]
	\[ \ \ \  \
	\rho(\beta_k,\xi)\ge|\widetilde{B}(\xi)|\ge |B(\xi)|-|\widetilde{B}(\xi)-B(\xi)|\ge  \theta-\eta\ge\delta=\delta_{0} ^{1-{1}/{\sqrt{-\ln\delta_{0}}}}.
	\]
    By Lemma \ref{ineq2}, we have for any $k\in\Lambda$,
	\begin{equation*}
		\rho(\alpha_k ,\xi)\ge\rho(\beta_k ,\xi)^{\gamma}.
	\end{equation*} 
	Furthermore, one can see that for any $\xi\in\partial\Omega$,
	\begin{equation*}
		|\widetilde{B}(\xi)|=\prod\limits_{k\in\Lambda}\rho(\beta_k,\xi)\le\prod\limits_{k\in\Lambda}\rho(\alpha_k,\xi)^{1/\gamma}=|B(\xi)|^{1/\gamma}=\theta^{1/\gamma}.
	\end{equation*}
	By the maximum modulus principle, $|\widetilde{B}(z)|\le\theta^{1/\gamma}$ holds for every $z\in\Omega$, which means the component $\Omega$ is contained in a component of the level set $\Omega_{\widetilde{B},\theta^{1/\gamma}}$. 

	Let 
	\[t'=\ln\frac{1+\theta^{1/\gamma}}{1-\theta^{1/\gamma}}.
	\] 
	By Lemma \ref{ineq3}, we have
    \begin{align*}
	t'-t&=~\ln\frac{1+\theta^{1/\gamma}}{1-\theta^{1/\gamma}}-\ln\frac{1+\theta}{1-\theta}\\
	&=~\ln\frac{1+\theta^{1/\gamma}}{1-\theta^{1/\gamma}}\cdot\frac{1-\theta}{1+\theta}\\
	&\le~\ln\gamma^{2}\\
	&=~2\ln\gamma,
    \end{align*}
	Then, each component of the level set $\Sigma_{B,t}$ is contained in a component of the level set $\Sigma_{\widetilde{B},t+2\ln\gamma}$. Since $t\ge\ln\frac{1+(\delta+\eta)}{1-(\delta+\eta)}\ge\ln\frac{1+\delta}{1-\delta}$, it follows from Claim 1 that $\alpha_i$ and $\alpha_j$ being in the same component of the level set $\Sigma_{B,t}$ implies $\beta_{i}$ and $\beta_j$ being in the same component of the level set $\Sigma_{\widetilde{B},t+2\ln\gamma}$. 
	In the same way, we could prove (2) of the Claim 2.
	
		\vspace{2mm}
	
	\textbf{Claim 3.}   Given any  $t\in(0, ~\ln\frac{1+(\delta+\eta)}{1-(\delta+\eta)})$.
	\begin{enumerate}
		\item[(1)] If $\alpha_i$ and $\alpha_j$ are contained in the same component of the level set $\Sigma_{B,t}$, then $\beta_{i}$ and $\beta_j$ are contained in the same component of the level set $\Sigma_{\widetilde{B},T_0}$.
		\item[(2)] 	 If $\beta_m$ and $\beta_n$ are contained in the same component of the level set $\Sigma_{\widetilde{B},t}$, then $\alpha_m$ and $\alpha_n$ are contained in the same component of the level set $\Sigma_{B,T_0}$.
	\end{enumerate}
	
	\vspace{2mm}
	
	Let $\Omega$ be an arbitrary component of the level set $\Sigma_{B, t}$. Notice that $\Sigma_{B, t}\subseteq\Omega_{B,\delta+\eta}$. Then, for any $\xi\in\partial\Omega$,
	\[
	|\widetilde{B}(\xi)|\le|B(\xi)|+|\widetilde{B}(\xi)-B(\xi)|\le\delta+\eta+\eta=\delta+2\eta.
	\]
	By the maximum modulus principle, $|\widetilde{B}(z)|\le\delta+2\eta$ holds for every $z\in\Omega$, which means the component $\Omega$ is contained in a component of the level set $\Omega_{\widetilde{B},\delta+2\eta}=\Sigma_{B,T_0}$. Since $T_0\ge\ln\frac{1+\delta}{1-\delta}$, it follows from Claim 1 that $\alpha_i$ and $\alpha_j$ being in the same component of the level set $\Sigma_{B,t}$ implies $\beta_{i}$ and $\beta_j$ being in the same component of the level set $\Sigma_{\widetilde{B},T_0}$. In the same way, we could prove (2) of  Claim 3.

	\vspace{2mm}
	
	Now, let us consider the persistence module induced by Blaschke products with strong Property $\mathfrak{B}$.
	For any $w\in\mathbb{D}$, denote by $[w]_{t}^{\mathbb{B}}$ the component of $\Sigma_{B,t}$ containing $w$, i.e.,
	\[
	[w]_{t}^{\mathbb{B}}=\{z;z\ \text{belongs}\ \text{to}\ \text{the}\ \text{same}\ \text{component}\ \text{of}\ \Sigma_{B,t}\ (\Omega_{B,\theta})\ \text{with}\ w\}.
	\] 
 Recall that  $[\alpha_{k}]_{t}^{\mathbb{B}}$ is the generator of zeroth homology group of the component of $\Sigma_{B,t}$ containing $\alpha_{k}$ and  $[\beta_{k}]_{t}^{\widetilde{\mathbb{B}}}$ is the generator of zeroth homology group of the component of $\Sigma_{\widetilde{B},t}$ containing $\beta_{k}$. Assume that the zero points of $B$ and $\widetilde{B}$ by $\{\alpha_{i}\}_{i\in\Lambda}$ and $\{\beta_{i}\}_{i\in\Lambda}$ are arranged as stated in Claim 1,
	we have
	\[
	V^{|B|}_{t}=\text{span}_{\mathbb{F}}\{[\alpha_{k}]_{t}^{\mathbb{B}}, \ k\in\Lambda\},
	\]
	and 
	\[
	V^{|\widetilde{B}|}_{t}=\text{span}_{\mathbb{F}}\{[\beta_{k}]_{t}^{\widetilde{\mathbb{B}}}, \ k\in\Lambda\}.
	\]

	For any $t>0$, define
	\[
	F([\alpha_{k}]_{t}^{\mathbb{B}})=[\beta_{k}]_{t+\Delta}^{\widetilde{\mathbb{B}}}, \ \ \ \text{for every} \ k\in\Lambda.
	\]
	For $t\in(0, ~\ln\frac{1+(\delta+\eta)}{1-(\delta+\eta)})$, it follows from Claim 3 that if $[\alpha_{i}]_{t}^{\mathbb{B}}=[\alpha_{j}]_{t}^{\mathbb{B}}$, then $[\beta_{i}]_{T_0}^{\widetilde{\mathbb{B}}}=[\beta_{j}]_{T_0}^{\widetilde{\mathbb{B}}}$ and hence $[\beta_{i}]_{t+\Delta}^{\widetilde{\mathbb{B}}}=[\beta_{j}]_{t+\Delta}^{\widetilde{\mathbb{B}}}$. 
	For $t\ge\ln\frac{1+(\delta+\eta)}{1-(\delta+\eta)}$, it follows from Claim 2 that if $[\alpha_{i}]_{t}^{\mathbb{B}}=[\alpha_{j}]_{t}^{\mathbb{B}}$, then  $[\beta_{i}]_{t+2\ln\gamma}^{\widetilde{\mathbb{B}}}=[\beta_{j}]_{t+2\ln\gamma}^{\widetilde{\mathbb{B}}}$ and hence $[\beta_{i}]_{t+\Delta}^{\widetilde{\mathbb{B}}}=[\beta_{j}]_{t+\Delta}^{\widetilde{\mathbb{B}}}$. Therefore, $F$ is well defined.

	Furthermore, for any $0<s<t$, it is obvious that for every $k\in\Lambda$,
	\[
	\widetilde{\pi}_{s+\Delta,t+\Delta}\circ F([\alpha_{k}]_{s}^{\mathbb{B}})=\widetilde{\pi}_{s+\Delta,t+\Delta}([\beta_{k}]_{s+\Delta}^{\widetilde{\mathbb{B}}})=[\beta_{k}]_{t+\Delta}^{\widetilde{\mathbb{B}}}=F([\alpha_{k}]_{t}^{\mathbb{B}})=F\circ\pi_{s,t}([\alpha_{k}]_{s}^{\mathbb{B}}).
	\]
    Then, we could naturally extend $F$ to be the persistence module homomorphism	
	 $F:\mathbb{B}\rightarrow\widetilde{\mathbb{B}}[\Delta]$.
	 
	Similarly, we define the persistence module homomorphism $G:\widetilde{\mathbb{B}}\rightarrow\mathbb{B}[\Delta]$ by
	\[
	G([\beta_{k}]_{t}^{\widetilde{\mathbb{B}}})=[\alpha_{k}]_{t+\Delta}^{\mathbb{B}}, \ \ \ \text{for every} \ k\in\Lambda.
	\]
	
	Notice that for any $t>0$, we have for every $k\in\Lambda$,
	\[
	G[\Delta]\circ F([\alpha_{k}]_{t}^{\mathbb{B}})=G[\Delta]([\beta_{k}]_{t+\Delta}^{\widetilde{\mathbb{B}}})=[\alpha_{k}]_{t+2\Delta}^{\mathbb{B}}=\Phi_{\mathbb{B}}^{2\Delta}([\alpha_{k}]_{t}^{\mathbb{B}}),
	\]
	and
	\[
	F[\Delta]\circ G([\beta_{k}]_{t}^{\widetilde{\mathbb{B}}})=F[\Delta]([\alpha_{k}]_{t+\Delta}^{\mathbb{B}})=[\beta_{k}]_{t+2\Delta}^{\widetilde{\mathbb{B}}}=\Phi_{\widetilde{\mathbb{B}}}^{2\Delta}([\beta_{k}]_{t}^{\widetilde{\mathbb{B}}}).
	\]
	
	Then, we obtain the following two commutative diagrams.
	\begin{center}
		\begin{tikzpicture}
			\node[inner sep=1pt] (a) at (0,0) {$\mathbb{B}$};
			\node[inner sep=1pt] (b) at (2,0) {$\widetilde{\mathbb{B}}[\Delta]$};
			\node[inner sep=1pt] (c) at (4,0) {$\mathbb{B}[2\Delta]$};
			\node at (1,-0.25) {$F$};
			\node at (2,0.95) {$\Phi_{\mathbb{B}}^{2\Delta}$};
			\node at (3,-0.25) {$G[\Delta]$};
			\draw[-latex] (a.0) -- (b.180);
			\draw[-latex] (b.0) -- (c.180);
			\draw[-latex] (a.40) to[out = 30,in=150]  (c.160);
		\end{tikzpicture}\ \ \ 
		\begin{tikzpicture}
			\node[inner sep=1pt] (a) at (0,0) {$\widetilde{\mathbb{B}}$};
			\node[inner sep=1pt] (b) at (2,0) {$\mathbb{B}[\Delta]$};
			\node[inner sep=1pt] (c) at (4,0) {$\widetilde{\mathbb{B}}[2\Delta]$};
			\node at (1,-0.25) {$G$};
			\node at (2,0.95) {$\Phi_{\widetilde{\mathbb{B}}}^{2\Delta}$};
			\node at (3,-0.25) {$F[\Delta]$};
			\draw[-latex] (a.0) -- (b.180);
			\draw[-latex] (b.0) -- (c.180);
			\draw[-latex] (a.40) to[out = 30,in=150]  (c.160);
		\end{tikzpicture}
	\end{center}
	Thus,
	\begin{equation*}
		d_{int}(\mathbb{B},\mathbb{\widetilde{B}})\le\Delta.
	\end{equation*} 
	\end{proof}

Now we are ready to prove Theorem B.

\begin{proof}[\bf{Proof of Theorem B}]
	Let $B$ be a Blaschke product with $2\eta_0$-strong Property $\mathfrak{B}$. Notice that for any $0<\eta<\eta_0$,  $B$ also has $2\eta$-strong Property $\mathfrak{B}$. Suppose that $\widetilde{B}$ is a Blaschke product with Property $\mathfrak{B}$ and $\|\widetilde{B}-B\|_{\infty}<\eta$. Then, by Lemma \ref{rhoalphabeta}, $\widetilde{B}$ has $\eta$-strong Property $\mathfrak{B}$. More precisely,
	\begin{equation*}
		\delta_{\widetilde{B},\eta}\le\delta_{B,2\eta},
	\end{equation*}
    this implies
	\begin{equation*}
	\delta_0(\eta)=\min\{\delta_{B,\eta},\delta_{\widetilde{B},\eta}\}\le\delta_{B,\eta}.
    \end{equation*}
    By Lemma \ref{limto0}, $\delta_0(\eta)\to 0$ as $\eta\to 0$. Furthermore, when $\eta\to 0$, we have 
    \[
    \gamma(\eta)=1+\frac{2}{\sqrt{-\ln\delta_{0}(\eta)}}\to 1^+,\ \ \ \ \ 
    \delta(\eta)=\delta_{0}(\eta) ^{1-{1}/{\sqrt{-\ln\delta_{0}(\eta)}}}\to 0, 
    \] 
    and
    \[
    T_0(\eta)=\ln\frac{1+(\delta(\eta)+2\eta)}{1-(\delta(\eta)+2\eta)}\to 0.
    \]
    Then by Lemma \ref{dintDelta}, $\eta\to 0$ implies
    \begin{equation*}
    	d_{int}(\mathbb{B},\widetilde{\mathbb{B}})\le\Delta(\eta)=\max\left\{T_0(\eta), ~2\ln\gamma(\eta)\right\}\to 0.
    \end{equation*}
   	Therefore, for any $\epsilon>0$, there exists $\eta\in(0,1)$ such that  for any Blaschke product $\widetilde{B}(z)$ with strong Property $\mathfrak{B}$,  $\|\widetilde{B}-B\|_{\infty}<\eta$ implies 
	\[
	d_{int}(\mathbb{B},\widetilde{\mathbb{B}})\le\Delta(\eta)<\epsilon.
	\]
    This completes the proof.
\end{proof}

    Theorem B demonstrates that the interleaving distance $d_{int}$ is continuous with respect to the superemum norm for Blaschke products  with strong Property $\mathfrak{B}$. As an application, we could use the the interleaving distance to characterize the connectedness in the space of all inner functions with strong Property $\mathfrak{B}$.

\begin{corollary}
	Let $B$ and $\widetilde{B}$ be two Blaschke products with strong Property $\mathfrak{B}$.  If there is a path connecting $B$ and $\widetilde{B}$ in the set of all inner functions with strong Property $\mathfrak{B}$, then the interleaving distance $d_{int}$ between $\mathbb{B}$ and $\widetilde{\mathbb{B}}$ is finite.
\end{corollary}
\begin{proof}
	Let $B_t$, $t\in[0,1]$, be a path connecting $B$ and $\widetilde{B}$ in the set of all inner functions with strong Property $\mathfrak{B}$. For any $t\in[0,1]$, it follows from Theorem B that there exists $\eta_t\in(0,1)$ such that , if $A$ is a Blaschke product with Property $\mathfrak{B}$ and $\|A-B_t\|_{\infty}<\eta_t$, then $A$ has $\eta_t$-strong Property $\mathfrak{B}$ and $d_{int}(\mathbb{A},\mathbb{B})\le1$. Denote
	\[
	U_t=\{A;~A \ \text{is a Blaschke product with Property} \ \mathfrak{B} \ \text{and} \ \|A-B_t\|_{\infty}<\eta_t\}.
	\]
	Then, $\{U_t\}_{t\in[0,1]}$ is an open cover of the path $B_t$ in set of all inner functions with strong Property $\mathfrak{B}$. Following from the compactness of $[0,1]$ and the  continuity of the path $B_t$ with respect to $t$,  there  exists a finite partition on $[0,1]$,
	\[
	0=t_0< t_1\cdots<t_k<\cdots <t_{K-1}<t_K=1,
	\]
	such that for any $k=0,1,\cdots, K-1$, both $B_{t_k}$ and $B_{t_{k+1}}$ belong to some certain open set $U_t$.
	Consequently, it follows from the triangle inequality of the interleaving distance that for any $k=0,1,\cdots, K-1$,
	\[
	d_{int}(\mathbb{B}_k,\mathbb{B}_{k+1})\le d_{int}(\mathbb{B}_k,\mathbb{B}_{t})+d_{int}(\mathbb{B}_t,\mathbb{B}_{k+1})\le2.
	\]
	Therefore, by the triangle inequality of the interleaving distance again, 
	$$
	d_{int}(\mathbb{B},\widetilde{\mathbb{B}})\le \sum\limits_{k=0}^{K-1}d_{int}(\mathbb{B}_k,\mathbb{B}_{k+1})\le 2K.
	$$
	This finishes the proof.
\end{proof}

By the way, the set of all inner functions with  weak Property $\mathfrak{H}$ contains all interpolating Blaschke products, and is contained in the set of all Carleson-Newman Blaschke products.
\begin{proposition}\label{interWPH}
	Each interpolating Blaschke product has  weak Property $\mathfrak{H}$.
\end{proposition}
\begin{proof}
	Let $B(z)=\lambda z^{m}\prod_{n}\frac{{\vert}z_{n}{\vert}}{z_{n}}\frac{z_{n}-z}{1-\overline{z_{n}}z}$ be an interpolating Blaschke product. Suppose that 
	\begin{equation*}
		\delta=\inf_{n\in \mathbb{N}}\prod_{k\neq n}{\big \vert}\frac{z_{k}-z_{n}}{1-\overline{z_{k}}z_{n}}{\big \vert}>0.
	\end{equation*}
   Then by Lemma \ref{Hoffman}, for any $0<\eta < (1-\sqrt{1-\delta^2})/\delta$ and
   	\[
   0<\epsilon<\eta\cdot\frac{\delta-\eta}{1-\delta\eta},
   \]
   $B$ has $\epsilon$-weak Property $\mathfrak{H}$.
\end{proof}

\begin{proposition}\label{CNWPH}
	Each Blaschke product with  weak Property $\mathfrak{H}$ is a Carleson-Newman Blaschke product.
\end{proposition}
\begin{proof}
    Let $B$ be a Blaschke product with $\eta$-weak Property $\mathfrak{H}$.
  By Lemma \ref{zero number}, there exists a positive integer $N$ such that each component of the level set $\Omega_{B,\eta}$ contains at most $N$ zero points. Then, we could write the Blaschke product $B$ as a finite product of Blaschke products $B_j$, $j=1,\cdots, N$, where each $B_j$ has at most one zero point in any one component of $\Omega_{B,\eta}$. It suffices to show that each $B_j$ is an interpolating Blaschke product. Given any $B_j$. Denote $\mathcal{Z}(B_j)=\{z_{j,k}\}_{k=1}^{\infty}$ and denote the  component of $\Omega_{B,\eta}$ containing $z_{j,k}$ by $\Omega_{j,k}$. For any $k\in\mathbb{N}$ and any $\xi\in\partial\Omega_{B,\eta}$,
  \[
 \left\vert \prod\limits_{i\neq k}\frac{z_{j,i}-\xi}{1-\overline{z_{j,i}}\xi} \right\vert \geq  \left\vert \prod\limits_{i=1}^{\infty}\frac{z_{j,i}-\xi}{1-\overline{z_{j,i}}\xi} \right\vert=\eta.
  \]
  Notice that the analytic function $\prod\limits_{i\neq k}\frac{z_{j,i}-\xi}{1-\overline{z_{j,i}}\xi}$ is non-zero everywhere in $\overline{\Omega_{j,k}}$. By the minimum modulus principle, 
   \[
  \left\vert \prod\limits_{i\neq k}\frac{z_{j,i}-z_{j,k}}{1-\overline{z_{j,i}}z_{j,k}}\right\vert\geq\inf\limits_{\xi\in\partial\Omega_{B,\eta}}\left\vert \prod\limits_{i\neq k}\frac{z_{j,i}-\xi}{1-\overline{z_{j,i}}\xi}\right\vert \geq\eta.
  \]
  Therefore, $B_j$ is an interpolating Blaschke product and hence $B$ is a Carleson-Newman Blaschke product.
\end{proof}

\section{Explicit formula of interleaving distance for Blaschke products with order two}

In this subsection, we consider Blashcke products with finitely many zeros. We provide an explicit formula for computing the interleaving distance between two Blaschke products with two zeros. For a fixed integer $n$, we also define a pseudo-metric on some certain spaces corresponding to the Blashcke products with $n$ zeros. When $n=2$, it will become a metric.

In the following context, the persistence module obtained  from a Blaschke product $B$ only cares about its modulus $|B|$. Without loss of generality, we take $B(z)=\prod_{n=1}^{+\infty}\frac{\beta_{n}-z}{1-\overline{\beta_{n}}z}$.

For the Blashcke products $B_{1}(z)$ and $B_{2}(z)$ with two  zeros, we can get the following theorem,

\begin{theorem}
Let $B_{1}$ be the Blashcke product with two zero points $\beta_{0}$ and $\beta_{1}$, and let $B_{2}$ be the Blashcke product with two zero points $\gamma_{0}$ and $\gamma_{1}$. Let $\mathbb{B}_1$ and $\mathbb{B}_2$ be the persistence modules induced by $B_1$ and $B_2$. Then, 
\[
d_{int}(\mathbb{B}_{1},\mathbb{B}_{2})=\min\left(\max(\frac{1}{2}\ln\frac{1}{\sqrt{1-|w_{2}|^{2}}},\frac{1}{2}\ln\frac{1}{\sqrt{1-|w_{1}|^{2}}}),|\ln\frac{\sqrt{1-|w_{1}|^{2}}}{\sqrt{1-|w_{2}|^{2}}}|\right),
\] 
where $w_{1}=\frac{\beta_{0}-\beta_{1}}{1-\overline{\beta_{0}}\beta_{1}}$, $w_{2}=\frac{\gamma_{0}-\gamma_{1}}{1-\overline{\gamma_{0}}\gamma_{1}}$.
\end{theorem}

\begin{proof}

	We divide the proof into three cases.
	\enumerate{
	
	\item  Suppose that $B_{1}$ and $B_{2}$ are the Blaschke products with two pairwise distinct zero points.
	
	 Let $\psi_{1}(z)=\frac{|w_{1}|}{w_{1}}\frac{\beta_{0}-z}{1-\overline{\beta_{0}}z}$, then $\mathbb{B}_{1}\cong \mathbb{B}_1(\psi_{1})$ by Lemma \ref{keeptr}, where $\mathbb{B}_1(\psi_{1})$ is the persistence module induced by $B_1\circ\psi_1$. And by Theorem A, we only need to know the critical points to reconstruct the persistence modules $\mathbb{B}_1$ and $\mathbb{B}_2$. Next, we compute the critical points of $B_1\circ\psi_1$. Denote by $c_1$   the critical point of $B_1\circ\psi_1$.
	
	Since the critical point $c_1$ satisfying that $(B_1\circ\psi_1(z))'|_{c_1}=0$,  following from 
	\begin{align*}
		B_1\circ\psi_1(z)=-z\cdot \frac{|w_1|-z}{1-|w_1|z},
	\end{align*} we have
	\begin{align*}
		\frac{1}{c_1}+\frac{|w_{1}|^2-1}{(1-|w_{1}|c_1)(|w_{1}|-c_1)}=0,
	\end{align*}
	and then,
	\begin{align*}
		c_1 = \frac{1-\sqrt{1-|w_{1}|^{2}}}{|w_{1}|}.
	\end{align*}
	Furthermore,
	\begin{align*}
		B_{1}\circ\psi_{1}(c_1)=-\frac{(1-\sqrt{1-|w_{1}|^{2}})^{2}}{|w_{1}|^{2}},
	\end{align*}
	\begin{align*}
		|B_{1}\circ\psi_{1}(c_1)|=\frac{(1-\sqrt{1-|w_{1}|^{2}})^{2}}{|w_{1}|^{2}}.
	\end{align*}
	Since
	\begin{align*}
		\ln\frac{1+|B_{1}\circ\psi_{1}(c_1)|}{1-|B_{1}\circ\psi_{1}(c_1)|}=\ln\frac{1}{\sqrt{1-|w_{1}|^{2}}},
	\end{align*}
	we have the persistence module $\mathbb{B}_{1}\cong\mathbb{F}(0,+\infty)\oplus \mathbb{F}(0,\ln\frac{1}{\sqrt{1-|w_{1}|^{2}}}]$.
	Let $w_{2}=\frac{\gamma_{0}-\gamma_{1}}{1-\overline{\gamma_{0}}\gamma_{1}}$, then $\mathbb{B}_{2}\cong\mathbb{F}(0,+\infty)\oplus \mathbb{F}(0,\ln\frac{1}{\sqrt{1-|w_{2}|^{2}}}]$. 
	
	Since two barcodes $\mathcal{B}$ and $\mathcal{C}$ are $\Delta$-matched with a finite $\Delta$ if and only if they have the same number of infinite bars (of form $(a,+\infty)$) and by Remark \ref{twodi}, we have $$d_{int}(\mathbb{B}_{1},\mathbb{B}_{2})=\min(\max(\frac{1}{2}\ln\frac{1}{\sqrt{1-|w_{2}|^{2}}},\frac{1}{2}\ln\frac{1}{\sqrt{1-|w_{1}|^{2}}}),|\ln\frac{\sqrt{1-|w_{1}|^{2}}}{\sqrt{1-|w_{2}|^{2}}}|).$$
	
	\item  Suppose that $B_{1}$ is the Blaschke products with zero point $\beta$ of order two and $B_{2}$ is the Blaschke products with zero point $\gamma$ of order two.
	
	 There is only one generator for each  persistence module.  By Theorem A,  $\mathbb{B}_{1}\cong\mathbb{F}(0,+\infty)$ and $\mathbb{B}_{2}\cong\mathbb{F}(0,+\infty)$, $$d_{int}(\mathbb{B}_{1},\mathbb{B}_{2})=0.$$
	
	\item Suppose that $B_{1}$ is the Blaschke products with two pairwise distinct zero points $\beta_{0}$ and $\beta_{1}$, $B_{2}$ is the Blaschke products with zero point $\gamma$ of order two.
	
	The persistence module  $\mathbb{B}_{1}\cong\mathbb{F}(0,+\infty)\oplus \mathbb{F}(0,\ln\frac{1}{\sqrt{1-|w_{1}|^{2}}}]$. And $\mathbb{B}_{2}\cong\mathbb{F}(0,+\infty)$, then we  have
	\begin{align*}
		d_{int}(\mathbb{B}_{1},\mathbb{B}_{2})=&\min(\frac{1}{2}\ln\frac{1}{\sqrt{1-|w_{1}|^{2}}},|\ln\frac{1}{\sqrt{1-|w_{1}|^{2}}}|)\\
		=&\frac{1}{2}\ln\frac{1}{\sqrt{1-|w_{1}|^{2}}}.
	\end{align*}
	
	Combining all three cases, we can summarize all cases in the following formula, $$d_{int}(\mathbb{B}_{1},\mathbb{B}_{2})=\min(\max(\frac{1}{2}\ln\frac{1}{\sqrt{1-|w_{2}|^{2}}},\frac{1}{2}\ln\frac{1}{\sqrt{1-|w_{1}|^{2}}}),|\ln\frac{\sqrt{1-|w_{1}|^{2}}}{\sqrt{1-|w_{2}|^{2}}}|).$$
	}

\end{proof}

For any {\bf a}=$(\alpha_1,\cdots,\alpha_n)\in \prod_{i=1}^{n} \mathbb{D}$, denote the Blaschke product with zero points $\{\alpha_1,\cdots,\alpha_n\}$ by $B_{\bf a}(z)$. We define $\bar{d}_n:\prod_{i=1}^{n} \mathbb{D}\times\prod_{i=1}^{n} \mathbb{D}\rightarrow\mathbb{R}$ by $\bar{d}_n({\bf a},{\bf b})=d_{int}(\mathbb{B}_{\bf a},\mathbb{B}_{\bf b})$. We say ${\bf a}\sim_{M}{\bf b}$ if there exists $\varphi \in \text{Aut}(\mathbb{D})$ such that ${\bf b}=\varphi({\bf a})=(\varphi(\alpha_1),\cdots,\varphi(\alpha_n))$.

\begin{proposition} \label{equ_dis}
	$\bar{d}_n(\cdot,\cdot)$ induces a pseudo-metric $d_n$ on $\prod_{i=1}^{n} \mathbb{D}/\sim_{M}$. 
\end{proposition}

\begin{proof}
	We define $d_n([{\bf a}],[{\bf b}])=\bar{d}_n({\bf a},{\bf b})$. By Lemma \ref{keeptr}, if  ${\bf a}\sim_{M}{\bf b}$ we have $$\bar{d}_n({\bf a},{\bf b})=d_{int}(\mathbb{B}_{\bf a},\mathbb{B}_{\bf b})=0,$$ which implies $d_n$ is well-defined.
	
	Since interleaving distance is a pseudo-metric, the symmetry and triangle inequality of $d_n$ follow from the property of $d_{int}$. 
\end{proof}

\begin{corollary}
	$d_2(\cdot,\cdot)$  is a metric on $\mathbb{D}\times\mathbb{D}/\sim_{M}$.
\end{corollary}

\begin{proof}
		For points ${\bf a}=(\beta_{0},\beta_{1})$, ${\bf b}=(\gamma_{0},\gamma_{1})$, if $d_2([{\bf a}],[{\bf b}])=0$, there are two cases. 
	
	In case 1, $$
	0=\max(\frac{1}{2}\ln\frac{1}{\sqrt{1-|w_{2}|^{2}}},\frac{1}{2}\ln\frac{1}{\sqrt{1-|w_{1}|^{2}}})\leq|\ln\frac{1}{\sqrt{1-|w_{2}|^{2}}}-\ln\frac{1}{\sqrt{1-|w_{1}|^{2}}}|,$$
	i.e., $|w_{2}|=|w_{1}|=0$, which means each of the Blaschke products $B_{\bf a}$ and $B_{\bf b}$ is with zero points of order two. There exists a  M\"{o}bius transformation $\psi$ such that $B_{\bf a}= B_{\bf b}\circ\psi$, then ${\bf a}\sim_{M} {\bf b}$.
	
	In case 2, $$0=|\ln\frac{1}{\sqrt{1-|w_{2}|^{2}}}-\ln\frac{1}{\sqrt{1-|w_{1}|^{2}}}|\leq \max(\frac{1}{2}\ln\frac{1}{\sqrt{1-|w_{2}|^{2}}},\frac{1}{2}\ln\frac{1}{\sqrt{1-|w_{1}|^{2}}}),$$
	i.e., $|w_{2}|=|w_{1}|$, there exists M\"{o}bius transformations $\psi_{1}=\frac{|w_{1}|}{w_{1}}\frac{\beta_{0}-z}{1-\overline{\beta_{0}}z}$ and $\psi_{2}=\frac{|w_{2}|}{w_{2}}\frac{\gamma_{0}-z}{1-\overline{\gamma_{0}}z}$ such that $B_{\bf a}\circ\psi_{1}=B_{\bf b}\circ\psi_{2}$, then ${\bf a}\sim_{M} {\bf b}$, which shows that $d_2$ is a metric.
\end{proof}

\section*{Declarations}

\noindent \textbf{Ethics approval}

\noindent Not applicable.

\noindent \textbf{Competing interests}

\noindent The authors declare that there is no conflict of interest or competing interest.

\noindent \textbf{Authors' contributions}

\noindent All authors contributed equally to this work.

\noindent \textbf{Funding}

\noindent The second author was partially supported by National Natural Science Foundation of China (Grant No. 12471120). The third author was partially supported by National Natural Science Foundation of China (Grant No. 11901229 and 12371029).

\noindent \textbf{Availability of data and materials}

\noindent Data sharing is not applicable to this article as no data sets were generated or analyzed during the current study.

\bibliographystyle{plain}
\bibliography{reference}

\end{document}